\documentclass{amsart}

\usepackage{array,booktabs,comment,delarray,float,rotating,titlesec}
\usepackage{amsmath,amsfonts,amssymb,amsthm}
\usepackage{bbm,bm,cases,dsfont,extarrows,mathrsfs,mathtools}
\usepackage{graphicx,xcolor}
\usepackage{pb-diagram,tikz,tikz-cd}
\usetikzlibrary{arrows,graphs}
\usepackage[all]{xy}
\usepackage[numbers]{natbib}
\usepackage{hyperref}

\makeatletter
\providecommand\@dotsep{4.5}
\renewcommand*\l@section[2]{\@dottedtocline{1}{1.5em}{3em}{#1}{#2}}
\renewcommand*\l@subsection[2]{\@dottedtocline{2}{3.0em}{3.8em}{#1}{#2}}
\makeatother

\titleformat{\section}{\normalfont\bfseries\Large}{\thesection}{1em}{}
\titlespacing*{\section}{0pt}{3.5ex plus 1ex minus .2ex}{2.3ex plus .2ex}

\titleformat{\subsection}{\normalfont\bfseries\large}{\thesubsection}{1em}{}
\titlespacing*{\subsection}{0pt}{2.5ex plus 1ex minus .2ex}{1.5ex plus .2ex}

\newtheorem{theorem}{Theorem}[section]
\newtheorem{lemma}[theorem]{Lemma}
\newtheorem{definition}[theorem]{Definition}
\newtheorem{corollary}[theorem]{Corollary}
\newtheorem{proposition}[theorem]{Proposition}
\newtheorem{example}[theorem]{Example}
\newtheorem{remark}[theorem]{Remark}
\newtheorem{conjecture}[theorem]{Conjecture}

\newtheorem*{claim}{Claim}

\newcommand{\red}[1]{#1}
\newcommand{\review}[1]{#1}
\newcommand{\revision}[1]{#1}

\renewcommand{\hom}{\operatorname{Hom}}
\newcommand{\ext}{\operatorname{Ext}}
\renewcommand{\mod}{\bmod}

\newcommand{\add}{\operatorname{add}}
\newcommand{\filt}{\operatorname{Filt}}

\newcommand{\id}{\mathbf{id}}
\newcommand{\cA}{\mathcal A}

\newcommand{\cT}{\mathcal T}
\newcommand{\cF}{\mathcal F}

\newcommand{\cJ}{\mathcal J}

\renewcommand{\ker}{{\rm Ker}}
\newcommand{\coker}{{\rm Coker}}
\newcommand{\im}{{\rm Im}}

\newcommand{\radlayers}[1]{\begin{smallmatrix}#1\end{smallmatrix}}

\newcommand{\bC}{\mathbb C}
\newcommand{\bQ}{\mathbb Q}

\newcommand{\bR}{\mathbb R}
\newcommand{\bF}{\mathbb F}
\newcommand{\bZ}{\mathbb Z}

\newcommand{\lra}{\longrightarrow}

\newcommand{\ra}{\rightarrow}
\newcommand{\sdp}{\times\kern-.2em\vrule height1.1ex depth-.05ex}
\newcommand{\epi}{\lra \kern-.8em\ra}

\makeatletter
\@addtoreset{equation}{section}
\makeatother

\title[Stability functions inducing stability data]{Stability functions inducing stability data}

\author{Fang Li}
\address{School of Mathematical Sciences, Zhejiang University, Yuhangtang Road 866, Hangzhou, Zhejiang 310058, China P.R.}
\email{fangli@zju.edu.cn}

\author{Kangping Qu}
\address{School of Mathematical Sciences, Zhejiang University, Yuhangtang Road 866, Hangzhou, Zhejiang 310058, China P.R.}
\email{qukp.zju@gmail.com}

\date{\today}

\begin{document}
\begin{abstract}
In this paper, we study when a stability data on an abelian length category is induced by a stability function. We give an intrinsic characterization of such stability data. For at most six phases, an inducing function exists if and only if every semistable subcategory is wide. An example with seven phases shows that this bound is sharp. For finite phase sets and finite rank, an inducing function exists if and only if an inducing red path exists. We also give a maximal green sequence of length eleven in type $D_4$ that no stability function induces. This disproves the red-path conjecture of Br\"ustle--Smith--Treffinger.
\end{abstract}

\maketitle

\tableofcontents
\section{Introduction}

Stability functions have their geometric origins in the theory of vector bundles on algebraic curves. In Mumford's construction of moduli spaces, stability is expressed by comparing the slope $\mu(E)=\deg(E)/\operatorname{rk}(E)$ of a vector bundle with the slopes of its proper nonzero subbundles \cite{Mumford1963}. The Narasimhan--Seshadri theorem gives this condition an analytic interpretation: a holomorphic vector bundle of degree zero on a compact Riemann surface is stable if and only if it arises from an irreducible unitary representation of the fundamental group \cite{NarasimhanSeshadri1965}. The additivity of degree and rank makes slope satisfy the see-saw property for short exact sequences of nonzero vector bundles, providing a model for the categorical notion of a stability function.

Stability functions assign values in a totally ordered set to the nonzero objects of an abelian category and satisfy the see-saw property for short exact sequences \cite{Rudakov1997}. On a length category, they determine semistable objects and Harder--Narasimhan (HN) filtrations. In 2004, Gorodentsev, Kuleshov and Rudakov defined stability data by specifying ordered semistable subcategories and requiring HN filtrations, without first choosing a stability function \cite[Definition~2.4]{GKR2004}. They observed that additive slopes give see-saw orderings, which in turn give stability data, and asked how much of the stronger structure can be recovered from the weaker one \cite[p.~753]{GKR2004}. We study the step from stability data to a see-saw ordering: given a stability data $(\Phi,\{\Pi_\varphi\}_{\varphi\in\Phi})$, when is there a stability function whose semistable objects are exactly the nonzero objects in the $\Pi_\varphi$, and whose value on each $\Pi_\varphi$ is constant and strictly increases with $\varphi$? We say that such a function \emph{induces} the stability data.

This question also concerns torsion classes. A stability data determines a chain of torsion classes, and a finest stability data with finitely many phases corresponds to a finite maximal green sequence in abelian length categories \cite{LQInfiniteMGS}. Chen, Lin and Ruan studied refinements of stability data and their relation to torsion pairs \cite{ChenLinRuan2023}. Br\"ustle, Smith and Treffinger showed that the semistable subcategories of a stability function are wide and studied maximal green sequences induced by stability functions and red paths \cite[Proposition~2.20 and Theorems~3.5 and~4.8]{BST}. Liu and Li characterized maximal green sequences induced by central charges through crossing inequalities \cite[Theorem~3.7]{LiLiu2020}. These results connect the question about given semistable subcategories to chains of torsion classes and to stability functions constructed from homomorphisms on $K_0(\cA)$.

The first result gives a criterion for arbitrary phase sets; $K_0(\cA)$ may also have infinite rank. For $\varphi\in\Phi$, write $H_\varphi$ for the $\bQ$-span of the classes of objects in $\Pi_\varphi$.

\begin{theorem}\label{thm:intro-general-criterion}\textup{\textbf{(Theorem~\ref{thm:positive-class-criterion})}}
  Let $\cA$ be an essentially small abelian length category, and let $(\Phi,\{\Pi_\varphi\}_{\varphi\in\Phi})$ be a stability data on $\cA$. The following conditions are equivalent:
  \begin{enumerate}
    \item A stability function into a totally ordered set induces $(\Phi,\{\Pi_\varphi\}_{\varphi\in\Phi})$.
    \item For every finite chain $\varphi_1<\cdots<\varphi_n$ in $\Phi$, the only tuple $\mathbf x_0,\ldots,\mathbf x_n\in K_0(\cA)_\bQ$ satisfying $\mathbf x_0=\mathbf x_n=0$, $\mathbf x_i$ in the nonnegative cone of $K_0(\cA)_\bQ$ for $1\leq i<n$, and $\mathbf x_i-\mathbf x_{i-1}\in H_{\varphi_i}$ for $1\leq i\leq n$ is the zero tuple.
    \item There are a totally ordered field $\bF$ and group homomorphisms $\theta_\varphi:K_0(\cA)\to\bF$ for $\varphi\in\Phi$ such that their $\bQ$-linear extensions vanish on $H_\varphi$ and $\theta_\varphi([S])>\theta_\psi([S])$ whenever $\varphi<\psi$ and $S$ is simple in $\cA$.
  \end{enumerate}
\end{theorem}

Condition~(2) is intrinsic to the given stability data: it uses only the phase order, the classes of objects in the semistable subcategories, and the nonnegative cone of $K_0(\cA)_\bQ$. In particular, the condition does not refer to a stability function or to a choice of ordered field. Condition~(3) expresses the same criterion through homomorphisms that vanish on the class spans $H_\varphi$ and decrease as the phase increases.

For at most six phases, the proof first constructs integer-valued homomorphisms $f,g$ with $f(H_i)=g(H_{i+1})=0$ and $f-g=\ell$ for two adjacent phases. It then constructs rational-valued homomorphisms for at most two phases on either side, preserving strict decrease on every simple class. This gives the homomorphisms in condition~(3). Br\"ustle, Smith and Treffinger proved that wideness of the semistable subcategories is necessary \cite[Proposition~2.20]{BST}; the next result shows that it is sufficient in this range.

\begin{theorem}\label{thm:intro-six-phase-criterion}\textup{\textbf{(Theorem~\ref{thm:small-phase-induction})}}
  Let $(\Phi,\{\Pi_\varphi\}_{\varphi\in\Phi})$ be a stability data on an essentially small abelian length category $\cA$, with $1\leq|\Phi|\leq6$. A stability function induces $(\Phi,\{\Pi_\varphi\}_{\varphi\in\Phi})$ if and only if every $\Pi_\varphi$ is wide. When every $\Pi_\varphi$ is wide, the stability function can be chosen with values in $\bQ$.
\end{theorem}

The bound of six phases is sharp. Example~\ref{ex:d4-seven-phase-counterexample} gives stability data on the category of finite-dimensional representations of a $D_4$ quiver over $\bC$ with seven semisimple wide semistable subcategories that are not induced by any stability function. Adding independent simple factors gives such examples with $n$ wide semistable subcategories for every finite $n>7$ (Remark~\ref{rem:seven-phase-extension}).

The six-phase result gives stability functions directly from wide semistable subcategories. In finite rank, red paths give another construction: Br\"ustle, Smith and Treffinger proved that every red path induces a stability function \cite[Theorem~4.8]{BST}. When $\Phi$ is finite and $K_0(\cA)$ has finite rank, an inducing stability function yields integer-valued homomorphisms whose values on each simple class strictly decrease with the phase (Corollary~\ref{cor:finite-phase-integral-separators}). After a common scaling, these values give consecutive vertices between $(1,\ldots,1)$ and $(-1,\ldots,-1)$. Joining the vertices by line segments gives a red path that induces the same stability data.

\begin{theorem}\label{thm:intro-red-paths}\textup{\textbf{(Corollary~\ref{cor:finite-induction-red-path})}}
  Let $(\Phi,\{\Pi_\varphi\}_{\varphi\in\Phi})$ be a stability data on an essentially small abelian length category $\cA$. If $\Phi$ is finite and $K_0(\cA)$ has finite rank, a stability function into a totally ordered set induces $(\Phi,\{\Pi_\varphi\}_{\varphi\in\Phi})$ if and only if a red path induces $(\Phi,\{\Pi_\varphi\}_{\varphi\in\Phi})$. When an inducing stability function exists, an inducing red path may be chosen piecewise linear, with rational vertices and strictly decreasing coordinates.
\end{theorem}

For finite-phase stability data in finite rank, the theorem gives the same existence criterion for stability functions and red paths. Br\"ustle, Smith and Treffinger proposed the following conjecture \cite[Conjecture~4.14]{BST}.

\begin{conjecture}\label{conj:intro-bst-red-path}\textup{\textbf{(Conjecture~\ref{conj:bst-red-path})}}
  Let $\cA$ be an abelian length category of finite rank. Every maximal green sequence in $\cA$ is induced by a red path.
\end{conjecture}

The seven-phase example above concerns stability data; answering Conjecture~\ref{conj:intro-bst-red-path} requires a maximal green sequence. The next example gives a maximal green sequence that no stability function induces.

\begin{example}\label{ex:intro-red-path-counterexample}\textup{\textbf{(Example~\ref{thm:d4-mgs-not-induced})}}
  For a $D_4$ quiver $Q$ over $K=\bC$, the category of finite-dimensional $KQ$-modules has a maximal green sequence of length $11$ that is not induced by any stability function into a totally ordered set. Consequently, no red path induces this sequence, and \cite[Conjecture~4.14]{BST} is false.
\end{example}

The semistable subcategories associated with this sequence are semisimple and wide. Thus the failure of induction also occurs for stability data whose torsion chain is a maximal green sequence.

The notation below is followed by the necessary definitions in Section~\ref{sec:preliminaries}. Section~\ref{sec:induction-criteria} proves the general criterion. Section~\ref{sec:wide-phase-bounds} establishes the six-phase result and its sharpness. Section~\ref{sec:mgs-red-paths} treats red paths and the counterexample to the conjecture.

\subsection*{Notation}
Unless stated otherwise, $\cA$ is an essentially small abelian length category. We use a fixed small skeleton when a set of objects is needed. All subcategories are full and closed under isomorphisms. For an abelian category $\mathcal H$, let $\operatorname{Obj}^*(\mathcal H)$ denote its nonzero objects. For subcategories $\mathcal B_1,\ldots,\mathcal B_s$ of $\cA$, $\filt(\mathcal B_1,\ldots,\mathcal B_s)$ denotes the objects admitting a finite filtration with factors in $\bigcup_{j=1}^s\mathcal B_j$. For a collection of objects $\mathcal B$, $\add\mathcal B$ consists of the direct summands of finite direct sums of objects in $\mathcal B$. For a finite-dimensional algebra $\Lambda$, $\mod\Lambda$ denotes its category of finite-dimensional left modules.

Write $\operatorname{Sim}\cA$ for a set of representatives of the isomorphism classes of simple objects of $\cA$. For $M\in\cA$ and $S\in\operatorname{Sim}\cA$, let $m_S(M)$ be the multiplicity of $S$ as a composition factor of $M$. We write $[M]$ for the class of $M$ in $K_0(\cA)$ and $\ell(M)$ for its composition length. For a homomorphism $f:K_0(\cA)\to G$ to an abelian group, $f(M)$ means $f([M])$. If $f,g:K_0(\cA)\to\bR$, we write $f>g$ when $f([S])>g([S])$ for every $S\in\operatorname{Sim}\cA$. Vector variables are boldface; their coordinates and linear functionals are not.

For a stability data with finitely many phases, write $\Phi=\{\varphi_1<\cdots<\varphi_n\}$. The term \emph{simple} is relative to the category under discussion. An \emph{ambient simple} is simple in $\cA$; an object simple in a semistable subcategory need not be ambient simple. We reserve $\phi(M)$ and $\phi^\pm(M)$ for phases determined by stability data and use $\mu$, with subscripts or primes, for stability functions. When $K_0(\cA)$ has finite rank, write $d$ for its rank; write $m$ for the length of a maximal green sequence. The number $n$ need not equal the number of values that an inducing stability function takes on all nonzero objects.

\section{Preliminaries}\label{sec:preliminaries}

\subsection{\texorpdfstring{\revision{Stability data and stability functions}}{Stability data and stability functions}}
An object of an abelian category has finite length if it admits a finite filtration with simple factors. An abelian length category is one in which every object has finite length. The Jordan--H\"older theorem gives
\begin{equation}\label{eq:grothendieck-simple-basis}
  K_0(\cA)=\bigoplus_{S\in\operatorname{Sim}\cA}\bZ[S],
  \qquad [M]=\sum_{S\in\operatorname{Sim}\cA}m_S(M)[S].
\end{equation}
The sum is finite, and $m_S(M)\geq0$. The homomorphism $\ell:K_0(\cA)\to\bZ$ defined by $\ell([S])=1$ satisfies $\ell(M)=\sum_Sm_S(M)$.

A full subcategory of an abelian category is \emph{wide} if it is closed under extensions and under kernels and cokernels of morphisms between its objects. It is \emph{Serre} if it is closed under subobjects, quotients and extensions. A wide subcategory of $\cA$ is an abelian length category. A Serre subcategory is wide, and an object in a Serre subcategory is simple there if and only if it is simple in the ambient category.

\begin{definition}\label{def:stability-data}\cite[Definition~2.4]{GKR2004}
  A stability data $(\Phi,\{\Pi_\varphi\}_{\varphi\in\Phi})$ on $\cA$ consists of a totally ordered set $\Phi$ and a nonzero extension-closed subcategory $\Pi_\varphi\subseteq\cA$ for each $\varphi\in\Phi$, such that:
  \begin{enumerate}
    \item $\hom(\Pi_\psi,\Pi_\varphi)=0$ whenever $\psi>\varphi$;
    \item every nonzero $M\in\cA$ admits a finite filtration
    \begin{equation}\label{eq:preliminary-HN-filtration}
      0=M_0\subsetneq M_1\subsetneq\cdots\subsetneq M_r=M
    \end{equation}
    with $M_j/M_{j-1}\in\Pi_{\psi_j}$ for $\psi_1>\cdots>\psi_r$ in $\Phi$.
  \end{enumerate}
\end{definition}

The filtration \eqref{eq:preliminary-HN-filtration} is the Harder--Narasimhan (HN) filtration. Its ordered phases and the isomorphism classes of its factors are uniquely determined. The nonzero objects in $\Pi_\varphi$ are \emph{semistable of phase} $\varphi$; for such an object $M$, put $\phi(M)=\varphi$. For any nonzero $M$ with the filtration \eqref{eq:preliminary-HN-filtration}, put $\phi^+(M)=\psi_1$ and $\phi^-(M)=\psi_r$. Its first HN factor $M_1$ is its maximal semistable subobject: every semistable subobject has phase at most $\phi^+(M)$, and every semistable subobject of that phase is contained in $M_1$. Dually, $M/M_{r-1}$ is its minimal semistable quotient. These conventions agree with \cite{LQInfiniteMGS}.

In the ordered-set formulation of Rudakov's stability structure \cite[Definition~1.1]{Rudakov1997}, a \emph{stability function} is an isomorphism-invariant map $\mu:\operatorname{Obj}^*(\cA)\to P$ to a totally ordered set $P$ such that every short exact sequence $0\to A\to B\to C\to0$ with all three terms nonzero satisfies exactly one of
\begin{equation}\label{eq:strict-see-saw-convention}
  \mu(A)<\mu(B)<\mu(C),\qquad
  \mu(A)=\mu(B)=\mu(C),\qquad
  \mu(A)>\mu(B)>\mu(C).
\end{equation}
This is the see-saw property. A nonzero object $M$ is $\mu$-semistable if $\mu(U)\leq\mu(M)$ for every nonzero subobject $U\subseteq M$ \cite[Definition~1.6]{Rudakov1997}. A stability function on a length category gives an HN filtration with $\mu$-semistable factors of strictly decreasing values \cite[Theorem~2.13]{BST}.

For $p\in P$, put
\begin{equation}
  \begin{aligned}
    \Pi_p^\mu&=\{0\}\cup\{M\in\operatorname{Obj}^*(\cA):\mu(M)=p,\ M\text{ is }\mu\text{-semistable}\},\\
    \Phi_\mu&=\{p\in P:\Pi_p^\mu\ne\{0\}\}.
  \end{aligned}
\end{equation}
With the order inherited from $P$, $(\Phi_\mu,\{\Pi_p^\mu\}_{p\in\Phi_\mu})$ is a stability data on $\cA$ \cite[Theorem~2.13]{BST}. We call this the stability data \emph{induced by} $\mu$. We say that a stability function $\mu:\operatorname{Obj}^*(\cA)\to P$ \emph{induces} a stability data $(\Phi,\{\Pi_\varphi\}_{\varphi\in\Phi})$ if there is an order isomorphism $\lambda:\Phi\to\Phi_\mu$ such that
\begin{equation}
  \Pi_\varphi=\Pi^\mu_{\lambda(\varphi)}\qquad(\varphi\in\Phi).
\end{equation}

A \emph{torsion pair} $(\cT,\cF)$ in an abelian category consists of full subcategories satisfying $\hom(\cT,\cF)=0$ such that every object $M$ admits a short exact sequence $0\to T_M\to M\to F_M\to0$ with $T_M\in\cT$ and $F_M\in\cF$. In a length category, a full subcategory containing zero is a torsion class if and only if it is closed under quotients and extensions. Dually, a torsion-free class is closed under subobjects and extensions. We use the conventions of \cite{LQInfiniteMGS}.

We first record how HN phases behave under subobjects, quotients and extensions.

\begin{lemma}\label{lem:HN-phase-bounds}
  \cite[Lemma~2.4]{LQInfiniteMGS} Let $\mathcal H$ be an abelian category with a stability data $(\Phi,\{\Pi_\varphi\}_{\varphi\in\Phi})$, and let $0\to A\to B\to C\to0$ be a short exact sequence in $\mathcal H$ with nonzero terms. Then:
  \begin{enumerate}
    \item $\phi^+(A)\leq\phi^+(B)$ and $\phi^-(C)\geq\phi^-(B)$;
    \item $\phi^+(B)\leq\max\{\phi^+(A),\phi^+(C)\}$ and $\phi^-(B)\geq\min\{\phi^-(A),\phi^-(C)\}$.
  \end{enumerate}
\end{lemma}

For a simple object in an abelian semistable subcategory, the subobject and quotient bounds become strict.

\begin{lemma}\label{lem:simple-slice-phase-bounds}
  Let $\mathcal H$ be an essentially small abelian category with a stability data $(\Phi,\{\Pi_\varphi\}_{\varphi\in\Phi})$. Fix $\alpha\in\Phi$. Suppose that $\Pi_\alpha$ is abelian and $X$ is simple in $\Pi_\alpha$. Then:
  \begin{enumerate}
    \item\label{item:simple-slice-subobjects} Every proper nonzero subobject of $X$ in $\mathcal H$ has all its HN phases strictly below $\alpha$.
    \item\label{item:simple-slice-quotients} Dually, every proper nonzero quotient of $X$ in $\mathcal H$ has all its HN phases strictly above $\alpha$.
  \end{enumerate}
\end{lemma}

\begin{proof}
  We prove (\ref{item:simple-slice-subobjects}); (\ref{item:simple-slice-quotients}) follows by passing to $\mathcal H^{\mathrm{op}}$ and reversing the phase order. Let $0\ne U\subsetneq X$ be a subobject in $\mathcal H$. By Lemma~\ref{lem:HN-phase-bounds}(1), $\phi^+(U)\leq\phi(X)=\alpha$. If equality holds, the maximal semistable subobject $U_1$ of $U$ in $\mathcal H$ belongs to $\Pi_\alpha$. The composite $U_1\to U\to X$ is then a nonzero monomorphism in $\Pi_\alpha$, hence an isomorphism by simplicity of $X$ in $\Pi_\alpha$. This forces $U=X$, a contradiction. Thus $\phi^+(U)<\alpha$.
\end{proof}

The next lemma shows that a stability function induces the given stability data if it has one value on each semistable subcategory and these values respect the phase order.

\begin{lemma}\label{lem:HN-recognition}
  Let $(\Phi,\{\Pi_\varphi\}_{\varphi\in\Phi})$ be a stability data on an abelian category $\mathcal H$, and let $\mu:\operatorname{Obj}^*(\mathcal H)\to P$ be a stability function, where $(P,\leq_P)$ is a totally ordered set. Suppose there are $\lambda_\varphi\in P$ for $\varphi\in\Phi$ such that $\mu(X)=\lambda_\varphi$ for every $0\ne X\in\Pi_\varphi$, and $\lambda_\varphi<\lambda_\psi$ whenever $\varphi<\psi$. Then $\mu$ induces $(\Phi,\{\Pi_\varphi\}_{\varphi\in\Phi})$.
\end{lemma}

\begin{proof}
  We first compare the value of an object with those of its HN factors.
  \begin{claim}
    If $M$ has HN factors $A_1,\ldots,A_s$ in decreasing phases, then
    \begin{equation}\label{eq:HN-value-bounds}
      \mu(A_s)\leq\mu(M)\leq\mu(A_1),
    \end{equation}
    with both inequalities strict when $s>1$.
  \end{claim}

  \begin{proof}[Proof of the claim]
    The first HN factor $A_1\subseteq M$ is its maximal semistable subobject in $\mathcal H$. Induction on $s$, using $0\to A_1\to M\to M/A_1\to0$ and the see-saw property, gives \eqref{eq:HN-value-bounds} and the asserted strictness.
  \end{proof}

  For $0\ne U\subsetneq X\in\Pi_\varphi$, Lemma~\ref{lem:HN-phase-bounds}(1) gives $\phi^+(U)\leq\varphi$. Applying the claim to $U$ therefore gives $\mu(U)\leq\mu(X)$, so $X$ is $\mu$-semistable. If $M$ has more than one HN factor, the claim shows that its first HN factor $A_1$ satisfies $\mu(A_1)>\mu(M)$; hence $M$ is not $\mu$-semistable.
\end{proof}

\subsection{\texorpdfstring{\revision{Maximal green sequences}}{Maximal green sequences}}

\begingroup
Recall that a torsion class is a full subcategory closed under quotients and extensions. A \emph{maximal green sequence} is a finite chain of torsion classes
\begin{equation}\label{eq:mgs-definition}
  0=\cT_0\subsetneq\cT_1\subsetneq\cdots\subsetneq\cT_m=\cA
\end{equation}
in which every inclusion is a cover, meaning that no torsion class lies strictly between $\cT_{i-1}$ and $\cT_i$. Let $\mu:\operatorname{Obj}^*(\cA)\to P$ be a stability function, with $P$ totally ordered. The \emph{$\mu$-HN factors} of a nonzero object $M$ are the successive quotients in its HN filtration; they are $\mu$-semistable and have strictly decreasing $\mu$-values. For each $p\in P$, put
\begin{equation}\label{eq:mgs-function-torsion}
  \cT^\mu_{\geq p}
  =\{0\}\cup\{M\ne0:\text{every $\mu$-HN factor of $M$ has value at least $p$}\}.
\end{equation}
\revision{The subcategories $\cT^\mu_{\geq p}$ are torsion classes} \cite[Propositions~2.17--2.18]{BST}, and $p\leq q$ implies $\cT^\mu_{\geq q}\subseteq\cT^\mu_{\geq p}$. When discussing an induced complete chain, we may enlarge $P$ by two elements $p_-,p_+$ with $p_-<p<p_+$ for every $p\in P$, neither attained by $\mu$. Then $\cT^\mu_{\geq p_-}=\cA$ and $\cT^\mu_{\geq p_+}=0$. This leaves all values and comparisons of $\mu$ unchanged. We say that $\mu$ \emph{induces} \eqref{eq:mgs-definition} if the distinct torsion classes obtained in this way are precisely $\cT_0,\ldots,\cT_m$.

\par\endgroup

\section{Criteria for inducing stability data}\label{sec:induction-criteria}

We first characterize when a stability data is induced by a stability function using admissible signings. \red{We then give a criterion in $K_0(\cA)_\bQ$ for an arbitrary phase set $\Phi$.}

\subsection{Induction via admissible signings}

Fix a stability data $(\Phi,\{\Pi_\varphi\}_{\varphi\in\Phi})$ on $\cA$. \review{Wideness of the semistable subcategories is necessary for a stability function to induce $(\Phi,\{\Pi_\varphi\}_{\varphi\in\Phi})$.}

\begin{proposition}
  \cite[Proposition 2.20]{BST}\label{prop:wideness of stability data} If a stability data $(\Phi,\{\Pi_\varphi\}_{\varphi\in\Phi})$ on $\cA$ is induced by a stability function $\mu:\operatorname{Obj}^*(\cA)\to P$ for a totally ordered set $(P,\leq_P)$, then each $\Pi_\varphi$ is a wide subcategory of $\cA$.
\end{proposition}

Example~\ref{ex:d4-seven-phase-counterexample} shows that wideness of all semistable subcategories does not ensure the existence of an inducing stability function. To state a necessary and sufficient condition, let
\begin{equation}
  \review{\mathscr E=\left\{(A,B,C)\in\operatorname{Obj}^*(\cA)^3\;\middle|\;
    \begin{gathered}
      \text{there exists a short exact sequence}\\[-2pt]
      0\to A\to B\to C\to0
    \end{gathered}\right\}.}
\end{equation}
For a stability function $\mu$ and $(A,B,C)\in\mathscr E$, exactly one of
\begin{equation}
  \mu(A)<\mu(B)<\mu(C),\qquad \mu(A)=\mu(B)=\mu(C),\qquad \mu(A)>\mu(B)>\mu(C)
\end{equation}
holds. \revision{We denote the increasing, constant and decreasing cases by $+$, $0$ and $-$, respectively.}

\begin{definition}\label{def:signing}
  Let $\mathscr F\subseteq\mathscr E$. Write $V_{\mathscr F}$ for the set of isomorphism classes of nonzero objects occurring in the triples of $\mathscr F$, using the fixed skeleton to represent these classes by objects. A \textbf{signing} on $\mathscr F$ is a map
  \begin{equation}
    \varepsilon:\mathscr F\to\{+,0,-\}.
  \end{equation}
  Define an equivalence relation $\sim_\varepsilon$ on $V_{\mathscr F}$ generated by the following relations, for all $X,Y\in V_{\mathscr F}$ and all $(A,B,C)\in\mathscr F$:
  \begin{equation}
    \begin{aligned}
      X\sim_\varepsilon Y &\quad\text{if }X,Y\in\Pi_\varphi\text{ for some }\varphi\in\Phi,\\
      A\sim_\varepsilon B\sim_\varepsilon C &\quad\text{if }\varepsilon(A,B,C)=0.
    \end{aligned}
  \end{equation}
  Let $Q_{\mathscr F,\varepsilon}$ be the quiver with vertex set $V_{\mathscr F}/\sim_\varepsilon$. Write $\revision{\overline X}$ for the $\sim_\varepsilon$-class of $X$. For $X,Y\in V_{\mathscr F}$, if $X\in\Pi_\varphi$ and $Y\in\Pi_\psi$ for phases $\varphi<\psi$ in $\Phi$, we add an arrow $\revision{\overline X}\to\revision{\overline Y}$. For $(A,B,C)\in\mathscr F$, we add
  \begin{equation}
    \revision{\overline A}\to\revision{\overline B}\to\revision{\overline C}\quad\text{if }\varepsilon(A,B,C)=+, \qquad \revision{\overline C}\to\revision{\overline B}\to\revision{\overline A}\quad\text{if }\varepsilon(A,B,C)=-.
  \end{equation}
  The signing $\varepsilon$ is called \textbf{admissible on $\mathscr F$} if $Q_{\mathscr F,\varepsilon}$ contains no directed cycle, where a loop is regarded as a cycle. When $\mathscr F=\mathscr E$, we simply call $\varepsilon$ an \textbf{admissible signing} and write $Q_{\mathscr E,\varepsilon}$ as $Q_\varepsilon$.
\end{definition}

\begin{definition}
  Let $(P,\leq_P)$ and $(P',\leq_{P'})$ be totally ordered sets, and let $\mu:\operatorname{Obj}^*(\cA)\to P$ and $\mu':\operatorname{Obj}^*(\cA)\to P'$ be stability functions. We say that $\mu$ and $\mu'$ are \emph{order-equivalent} if, for all nonzero objects $X,Y\in\cA$,
  \begin{equation}
    \begin{aligned}
      \mu(X)<\mu(Y)&\Longleftrightarrow\mu'(X)<\mu'(Y),\\
      \mu(X)=\mu(Y)&\Longleftrightarrow\mu'(X)=\mu'(Y).
    \end{aligned}
  \end{equation}
\end{definition}

The following proposition relates admissible signings to the equality and strict order comparisons defined by inducing stability functions.

\begin{proposition}\label{prop:stability-induction-criterion}
  Let $(\Phi,\{\Pi_\varphi\}_{\varphi\in\Phi})$ be a system of stability data on an abelian length category $\cA$. The following statements are equivalent:
  \begin{enumerate}
    \item \revision{there is a stability function into a totally ordered set inducing $(\Phi,\{\Pi_\varphi\}_{\varphi\in\Phi})$;}
    \item there exists an admissible signing on $\mathscr E$.
  \end{enumerate}
  Moreover, $\mu\mapsto\varepsilon_\mu$ induces a bijection between the order-equivalence classes of stability functions inducing $(\Phi,\{\Pi_\varphi\}_{\varphi\in\Phi})$ and the admissible signings on $\mathscr E$.
\end{proposition}

\begin{proof}
  Assume first that a stability function $\mu$ as in (1) is given. Define
  \begin{equation}
    \varepsilon_\mu(A,B,C)=
    \begin{cases}
      +, & \mu(A)<\mu(B)<\mu(C),\\
      0, & \mu(A)=\mu(B)=\mu(C),\\
      -, & \mu(A)>\mu(B)>\mu(C).
    \end{cases}
  \end{equation}
  The strict see-saw property shows that $\varepsilon_\mu$ is well defined. Every relation generating $\sim_{\varepsilon_\mu}$ gives equality of $\mu$-values, while every arrow of $Q_{\varepsilon_\mu}$ gives a strict inequality. Hence a directed cycle would give $\mu(X)<\mu(X)$ for some nonzero object $X$, which is impossible. Thus $\varepsilon_\mu$ is admissible.

  Conversely, let $\varepsilon$ be an admissible signing and put $P_\varepsilon=V_{\mathscr E}/\sim_\varepsilon$. For $\revision{\overline X},\revision{\overline Y}\in P_\varepsilon$, define $\revision{\overline X}<_\varepsilon\revision{\overline Y}$ if $Q_\varepsilon$ contains a directed path of positive length from $\revision{\overline X}$ to $\revision{\overline Y}$. Since $Q_\varepsilon$ is acyclic, this is a strict partial order. It is in fact total. Indeed, for any nonzero objects $X,Y$, consider the split short exact sequence $0\to X\to X\oplus Y\to Y\to0$. According to whether its sign is $+$, $0$, or $-$, we obtain respectively $\revision{\overline X}<\revision{\overline Y}$, $\revision{\overline X}=\revision{\overline Y}$, or $\revision{\overline X}>\revision{\overline Y}$.

  Define $\mu_\varepsilon:\operatorname{Obj}^*(\cA)\to P_\varepsilon$ by $\mu_\varepsilon(X)=\revision{\overline X}$. By construction, $\mu_\varepsilon$ satisfies the strict see-saw property, is constant on each $\Pi_\varphi\setminus\{0\}$, and satisfies $\mu_\varepsilon(X)<\mu_\varepsilon(Y)$ whenever $0\ne X\in\Pi_\varphi$, $0\ne Y\in\Pi_\psi$, and $\varphi<\psi$. Lemma~\ref{lem:HN-recognition} therefore shows that it induces $(\Phi,\{\Pi_\varphi\}_{\varphi\in\Phi})$.

  Finally, $\mu\mapsto\varepsilon_\mu$ depends only on the order-equivalence class of $\mu$, and the construction gives $\varepsilon_{\mu_\varepsilon}=\varepsilon$, proving surjectivity. If $\varepsilon_\mu=\varepsilon_{\mu'}$, then the sign of the split sequence $0\to X\to X\oplus Y\to Y\to0$ gives the same comparison between $\mu(X)$ and $\mu(Y)$ as between $\mu'(X)$ and $\mu'(Y)$ for every pair $X,Y$. Hence $\mu$ and $\mu'$ are order-equivalent, proving injectivity.
\end{proof}

The signing criterion can be checked on finite collections of short exact sequences.

\begin{proposition}\label{prop:reduction to finite case}
  Let $(\Phi,\{\Pi_\varphi\}_{\varphi\in\Phi})$ be a system of stability data on $\cA$. Suppose that every finite subset $\mathscr F\subseteq\mathscr E$ admits an admissible signing. Then $\mathscr E$ itself admits an admissible signing. In particular, there exist a totally ordered set $(P,\leq_P)$ and a stability function $\mu:\operatorname{Obj}^*(\cA)\to P$ inducing $(\Phi,\{\Pi_\varphi\}_{\varphi\in\Phi})$.
\end{proposition}

\begin{proof}
  Let
  \begin{equation}
    \Omega=\{+,0,-\}^{\mathscr E}
  \end{equation}
  be the space of all signings on $\mathscr E$, equipped with the product topology of discrete spaces. By Tychonoff's theorem, $\Omega$ is compact. For every finite $\mathscr F\subseteq\mathscr E$, set
  \begin{equation}
    C_{\mathscr F} = \{\varepsilon\in\Omega\mid \varepsilon|_{\mathscr F}\text{ is admissible on }\mathscr F\}.
  \end{equation}
  Each $C_{\mathscr F}$ is clopen, since membership depends only on the finitely many coordinates indexed by $\mathscr F$.

  For finite subsets $\mathscr F_1,\ldots,\mathscr F_r\subseteq\mathscr E$, put $\mathscr G=\bigcup_{j=1}^r\mathscr F_j$. By hypothesis, $\mathscr G$ admits an admissible signing; extending it arbitrarily to $\mathscr E$ shows that $\varnothing\ne C_{\mathscr G}\subseteq\bigcap_{j=1}^rC_{\mathscr F_j}$. Thus the family $\{C_{\mathscr F}\mid\mathscr F\subseteq\mathscr E\text{ finite}\}$ has the finite intersection property. By compactness,
  \begin{equation}
    \bigcap_{\substack{\mathscr F\subseteq\mathscr E\\ \mathscr F\text{ finite}}} C_{\mathscr F}\neq\varnothing.
  \end{equation}
  Choose $\varepsilon$ in this intersection.

  We claim that $\varepsilon$ is admissible. Otherwise, $Q_\varepsilon$ contains a cycle. This cycle involves only finitely many arrows, and its vertex identifications are generated by finitely many same-phase relations and finitely many triples signed $0$. Choose a finite subset $\mathscr F\subseteq\mathscr E$ containing all triples required for these arrows and identifications. If an object in the cycle does not belong to $V_{\mathscr F}$, add the triple $(X,X\oplus X,X)$, corresponding to the split exact sequence $0\to X\to X\oplus X\to X\to0$, to $\mathscr F$. The same cycle then occurs in $Q_{\mathscr F,\varepsilon|_{\mathscr F}}$, contradicting $\varepsilon\in C_{\mathscr F}$.

  Thus $\varepsilon$ is admissible. The final assertion follows from Proposition~\ref{prop:stability-induction-criterion}.
\end{proof}

\begin{remark}
  In Proposition~\ref{prop:reduction to finite case}, the signings $\varepsilon_{\mathscr F}$ chosen on finite subsets are not assumed to satisfy $\varepsilon_{\mathscr G}|_{\mathscr F} =\varepsilon_{\mathscr F}$ when $\mathscr F\subseteq\mathscr G$.
\end{remark}

\subsection{A criterion for stability data via Grothendieck group}\label{subsec:positive-class-criterion}

To pass from exact sequences to Grothendieck classes, we first show that a stability function on a length category depends only on composition multiplicities.

\begin{lemma}\label{lem:stability-splitting-neutrality}
  Let $\cA$ be an abelian category, let $(P,\leq_P)$ be a totally ordered set, and let $\mu:\operatorname{Obj}^*(\cA)\to P$ be a stability function. For every short exact sequence $0\to A\to E\to C\to0$ with $E\ne0$, one has $\mu(E)=\mu(A\oplus C)$. If $\cA$ is a length category, then $\mu(M)=\mu(N)$ for all $M,N\in\operatorname{Obj}^*(\cA)$ with $m_S(M)=m_S(N)$ for every $S\in\operatorname{Sim}\cA$.
\end{lemma}

\begin{proof}
  If $A=0$ or $C=0$, then $E\simeq A\oplus C$. Assume $A,C\ne0$ and write the maps as $i:A\to E$ and $p:E\to C$. There are short exact sequences
  \begin{equation}\label{eq:stability-neutralizing-rows}
    \begin{aligned}
      0&\longrightarrow A\oplus C
      \xrightarrow{\left(\begin{smallmatrix}i&0\\0&0\\0&\id_C\end{smallmatrix}\right)} E\oplus A\oplus C
      \xrightarrow{\left(\begin{smallmatrix}p&0&0\\0&\id_A&0\end{smallmatrix}\right)} C\oplus A
      \longrightarrow0,\\
      0&\longrightarrow E\longrightarrow E\oplus A\oplus C \longrightarrow A\oplus C\longrightarrow0.
    \end{aligned}
  \end{equation}
  The first sequence in \eqref{eq:stability-neutralizing-rows} has isomorphic endpoints, so strict see-saw gives $\mu(E\oplus A\oplus C)=\mu(A\oplus C)$. Applying see-saw to the second sequence in \eqref{eq:stability-neutralizing-rows} forces $\mu(E)=\mu(A\oplus C)$. For an additional object $R$, apply the same argument to $0\to A\oplus R\to E\oplus R\to C\to0$, where the first map is $i\oplus\id_R$ and the second is $(p,0)$. This gives $\mu(E\oplus R)=\mu(A\oplus C\oplus R)$. Applying this equality successively to a composition series of $M\ne0$ gives $\mu(M)=\mu\bigl(\bigoplus_{S\in\operatorname{Sim}\cA} S^{\oplus m_S(M)}\bigr)$. This proves the last assertion.
\end{proof}

For the rest of this section, fix a stability data $(\Phi,\{\Pi_\varphi\}_{\varphi\in\Phi})$ on $\cA$, \red{where $\Phi$ is an arbitrary nonempty totally ordered set}. Put
\begin{equation}\label{eq:phase-class-spaces}
  \begin{aligned}
    K_0(\cA)_\bQ&=\bigoplus_{S\in\operatorname{Sim}\cA}\bQ[S], &K_0(\cA)_{\bQ_{\geq0}}&=\bigoplus_{S\in\operatorname{Sim}\cA}\bQ_{\geq0}[S],\\
    \red{H_\varphi}&=\red{\operatorname{span}_\bQ\{[E]\mid E\in\Pi_\varphi\}\quad(\varphi\in\Phi).}
  \end{aligned}
\end{equation}
\red{For $\mathbf x=\sum_S x_S[S]\in K_0(\cA)_\bQ$, put $\operatorname{supp}(\mathbf x)=\{S\in\operatorname{Sim}\cA:x_S\ne0\}$; then $|\operatorname{supp}(\mathbf x)|<\infty$.}

The order on a totally ordered field $\bF$ satisfies $a<b\implies a+c<b+c$ and $a,b>0\implies ab>0$ for $a,b,c\in\bF$. In particular, $m\cdot1>0$ for every integer $m\geq1$, so $\operatorname{char}\bF=0$. We identify $\bQ$ with its image under the canonical order-preserving embedding $\bQ\hookrightarrow\bF$, and write $[a,b]_\bF=\{t\in\bF\mid a\leq t\leq b\}$ for $a\leq b$ in $\bF$. We use the same notation for a group homomorphism $K_0(\cA)\to\bF$ and its $\bQ$-linear extension to $K_0(\cA)_\bQ$.

We will construct stability functions as the unique zeros of families of additive maps. \revision{The following criterion supplies the see-saw property.}

\begin{proposition}\label{prop:red-path-criterion}
  \review{Adapted from the proof of \cite[Theorem~4.8]{BST}.} Let $\cA$ be an essentially small abelian length category, let $\bF$ be a totally ordered field, and let $I\subseteq\bF$ be an interval with the induced order. Let $F_t:K_0(\cA)\to\bF$ be group homomorphisms for $t\in I$. Suppose that, for every nonzero $M\in\cA$, there is $\mu(M)\in I$ such that, for all $t\in I$,
  \begin{equation}
    \begin{aligned}
      F_t(M)&>0&&\quad\text{if }t<\mu(M),\\
      F_t(M)&=0&&\quad\text{if }t=\mu(M),\\
      F_t(M)&<0&&\quad\text{if }t>\mu(M).
    \end{aligned}
  \end{equation}
  Then $\mu:\operatorname{Obj}^*(\cA)\to I$ is a stability function.
\end{proposition}

\begin{proof}
  Since $F_t(M)$ depends only on $[M]$ and has a unique zero in $t$, $\mu$ is isomorphism-invariant. Take a short exact sequence $0\to L\to M\to N\to0$ of nonzero objects and set $t=\mu(M)$. Additivity gives $F_t(L)+F_t(N)=F_t(M)=0$. If $\mu(L)<t$, the sign condition gives $F_t(L)<0<F_t(N)$, hence $\mu(L)<\mu(M)<\mu(N)$. If $\mu(L)=t$, then $F_t(L)=F_t(N)=0$, so $\mu(N)=t$; if $\mu(L)>t$, then $F_t(L)>0>F_t(N)$, so $\mu(N)<t$. Thus $\mu$ satisfies the see-saw property.
\end{proof}

The main criterion relates the existence of an inducing stability function to positive Grothendieck classes and to decreasing homomorphisms vanishing on the semistable subcategories.

\begin{theorem}\label{thm:positive-class-criterion}
  Let $\cA$ be an essentially small abelian length category, and let $(\Phi,\{\Pi_\varphi\}_{\varphi\in\Phi})$ be a stability data on $\cA$, \red{where $\Phi$ is an arbitrary nonempty totally ordered set}. \red{For $\varphi\in\Phi$, put $H_\varphi=\operatorname{span}_\bQ\{[E]\mid E\in\Pi_\varphi\}\subseteq K_0(\cA)_\bQ$.} The following conditions are equivalent.
  \begin{enumerate}
    \item There exist a totally ordered set $(P,\leq_P)$ and a stability function $\mu:\operatorname{Obj}^*(\cA)\to P$ inducing $(\Phi,\{\Pi_\varphi\}_{\varphi\in\Phi})$.
    \item \red{For every finite chain $\varphi_1<\cdots<\varphi_n$ in $\Phi$, put $H_i=H_{\varphi_i}$.} For vectors $\mathbf x_0,\ldots,\mathbf x_n\in K_0(\cA)_\bQ$, the system
    \begin{equation}\label{eq:positive-class-flow}
      \left\{\begin{aligned}
        \mathbf x_0=\mathbf x_n&=\mathbf 0,\\
        \mathbf x_i&\in K_0(\cA)_{\bQ_{\geq0}} &&(1\leq i<n),\\
        \mathbf x_i-\mathbf x_{i-1}&\in H_i &&(1\leq i\leq n)
      \end{aligned}\right.
    \end{equation}
    has only the solution $\mathbf x_0=\cdots=\mathbf x_n=\mathbf 0$.
    \item There are a totally ordered field $\bF$ and group homomorphisms \red{$\theta_\varphi:K_0(\cA)\to\bF$ $(\varphi\in\Phi)$} such that \red{$\theta_\varphi(H_\varphi)=0$ for every $\varphi\in\Phi$}, and \red{$\theta_\varphi([S])>\theta_\psi([S])$ whenever $\varphi<\psi$} and $S\in\operatorname{Sim}\cA$.
  \end{enumerate}
\end{theorem}

We first prove (1)$\Rightarrow$(2). \red{For (2)$\Rightarrow$(1), we apply Proposition~\ref{prop:reduction to finite case}; for (2)$\Leftrightarrow$(3), Proposition~\ref{lem:finite-dimensional-separation} gives rational solutions of the finite systems, and Lemma~\ref{lem:rational-system-compactness} gives a simultaneous solution over a totally ordered field.}

\begin{lemma}\label{lem:positive-class-obstruction}
  Let $(P,\leq_P)$ be a totally ordered set, and let $\mu:\operatorname{Obj}^*(\cA)\to P$ induce $(\Phi,\{\Pi_\varphi\}_{\varphi\in\Phi})$. Then condition~(2) of Theorem~\ref{thm:positive-class-criterion} holds.
\end{lemma}

\begin{proof}
  \red{Fix a finite chain $\varphi_1<\cdots<\varphi_n$ in $\Phi$, and put $H_i=H_{\varphi_i}$.} For $1\leq i\leq n$, let $\lambda_i\in P$ be the constant value of $\mu$ on $\Pi_{\varphi_i}\setminus\{0\}$; then $\lambda_1<\cdots<\lambda_n$. Suppose $(\mathbf x_0,\ldots,\mathbf x_n)$ is a nonzero solution of \eqref{eq:positive-class-flow}. Write each difference as a finite sum $\mathbf x_i-\mathbf x_{i-1}=\sum_r c_{ir}[A_{ir}]$, where $c_{ir}\in\bQ$ and $A_{ir}\in\Pi_{\varphi_i}$. Multiplying all $\mathbf x_i$ and $c_{ir}$ by a common positive integer, we may assume $c_{ir}\in\bZ$ and $\mathbf x_i=\sum_S a_{i,S}[S]$ with $a_{i,S}\in\bZ_{\geq0}$. Set $Z_i=\bigoplus_S S^{\oplus a_{i,S}}$ for $1\leq i<n$ and $Z_0=Z_n=0$. Define
  \begin{equation}
    E_i=\bigoplus_{r:c_{ir}>0}A_{ir}^{\oplus c_{ir}}, \qquad F_i=\bigoplus_{r:c_{ir}<0}A_{ir}^{\oplus(-c_{ir})}.
  \end{equation}
  Then $E_i,F_i\in\Pi_{\varphi_i}$ and $\mathbf x_i-\mathbf x_{i-1}=[E_i]-[F_i]$. Without loss of generality, we can assume $E_i,F_i\ne0$ for each $i$. Otherwise, one can choose $0\ne W_i\in\Pi_{\varphi_i}$ and replace $(E_i,F_i)$ by $(E_i\oplus W_i,F_i\oplus W_i)$. Moreover,
  \begin{equation}\label{eq:positive-class-balances}
    [E_i\oplus Z_{i-1}]=[F_i\oplus Z_i]\qquad(1\leq i\leq n).
  \end{equation}
  Let $a=\min\{i:Z_i\ne0\}$; then $1\leq a<n$. Since $Z_{a-1}=0$, Lemma~\ref{lem:stability-splitting-neutrality} gives $\mu(F_a\oplus Z_a)=\mu(E_a)=\lambda_a=\mu(F_a)$, so see-saw implies $\mu(Z_a)=\lambda_a<\lambda_{a+1}$. We now induct on $i=a+1,\ldots,n$. Suppose $Z_{i-1}\ne0$ and $\mu(Z_{i-1})<\lambda_i$. Strict see-saw and Lemma~\ref{lem:stability-splitting-neutrality} give $\mu(F_i\oplus Z_i)=\mu(E_i\oplus Z_{i-1})<\lambda_i$. Since $\mu(F_i)=\lambda_i$, we have $Z_i\ne0$; strict see-saw then gives $\mu(Z_i)<\lambda_i$. For $i<n$, this implies $\mu(Z_i)<\lambda_{i+1}$, so induction yields $Z_n\ne0$, contradicting $Z_n=0$.
\end{proof}

The next step constructs the required homomorphisms for a finite-dimensional rational system.

For the following proposition, let $n\geq1$ and $d\geq0$ be integers, and let $H_1,\ldots,H_n$ be $\bQ$-linear subspaces of $\bQ^d$. Write $\mathbf e_1,\ldots,\mathbf e_d$ for the standard basis, and choose a basis $\mathbf h_{i,1},\ldots,\mathbf h_{i,r_i}$ of each $H_i$. Consider the system in $\bQ$-linear maps $\theta_1,\ldots,\theta_n:\bQ^d\to\bQ$
\begin{equation}\label{eq:rational-cut-system}
  \left\{\begin{aligned}
    \theta_i(\mathbf h_{i,k})&=0 &&(1\leq i\leq n,\ 1\leq k\leq r_i),\\
    \theta_i(\mathbf e_j)-\theta_{i+1}(\mathbf e_j)&\geq1 &&(1\leq i<n,\ 1\leq j\leq d).
  \end{aligned}\right.
\end{equation}

\begin{proposition}\label{lem:finite-dimensional-separation}
  The following conditions are equivalent.
  \begin{enumerate}
    \item If $\mathbf x_0,\ldots,\mathbf x_n\in\bQ_{\geq0}^d$ satisfy $\mathbf x_0=\mathbf x_n=\mathbf0$ and $\mathbf x_i-\mathbf x_{i-1}\in H_i$ for $1\leq i\leq n$, then $\mathbf x_0=\cdots=\mathbf x_n=\mathbf0$.
    \item \eqref{eq:rational-cut-system} has a rational solution.
  \end{enumerate}
  In (2), the coefficients $\theta_i(\mathbf e_j)$ can all be chosen in $\bZ$.
\end{proposition}

\begin{proof}
  If $d=0$ or $n=1$, both conditions hold with $\theta_i=0$. Assume $d>0$ and $n>1$, and put $N=d(n-1)$. For a subspace $L\subseteq\bQ^m$, let $L^\perp$ denote its orthogonal complement with respect to the standard inner product, denoted by $\langle\cdot,\cdot\rangle$. Using this inner product, we identify $(\bQ^d)^*$ with $\bQ^d$, with $\theta_i$ corresponding to its coefficient vector
  \begin{equation}
    \bm\theta_i=(\theta_i(\mathbf e_1),\ldots,\theta_i(\mathbf e_d)), \qquad \theta_i(\mathbf x)=\langle\bm\theta_i,\mathbf x\rangle.
  \end{equation}
  Thus $\theta_i(H_i)=0$ if and only if $\bm\theta_i\in H_i^\perp$. For $\bm\theta=(\bm\theta_1,\ldots,\bm\theta_n)$, define
  \begin{equation}
    \Delta:\bigoplus_{i=1}^n H_i^\perp\longrightarrow\bQ^N,\qquad \Delta(\bm\theta)=(\bm\theta_1-\bm\theta_2,\ldots, \bm\theta_{n-1}-\bm\theta_n).
  \end{equation}
  Put $W=\operatorname{im}\Delta$. For $\mathbf x_1,\ldots,\mathbf x_{n-1}\in\bQ^d$ and $\mathbf x_0=\mathbf x_n=\mathbf0$, we have
  \begin{equation}\label{eq:class-flow-duality}
    \sum_{i=1}^{n-1}\langle\bm\theta_i-\bm\theta_{i+1},\mathbf x_i\rangle =\sum_{i=1}^n\langle\bm\theta_i,\mathbf x_i-\mathbf x_{i-1}\rangle.
  \end{equation}
  Since each $\bm\theta_i$ ranges independently over $H_i^\perp$ and $(H_i^\perp)^\perp=H_i$, this identity gives
  \begin{equation}
    (\mathbf x_1,\ldots,\mathbf x_{n-1})\in W^\perp \quad\Longleftrightarrow\quad \mathbf x_i-\mathbf x_{i-1}\in H_i\quad(1\leq i\leq n).
  \end{equation}
  Hence (1) is equivalent to $W^\perp\cap\bQ_{\geq0}^N=\{\mathbf0\}$. The Gordan--Stiemke theorem over $\bQ$ (see, e.g., \cite[Lemma~1.2, p.~530]{SaundersSchneider1979}) gives
  \begin{equation}\label{eq:orthant-alternative}
    W\cap\bQ_{>0}^N\ne\varnothing \quad\Longleftrightarrow\quad W^\perp\cap\bQ_{\geq0}^N=\{\mathbf0\}.
  \end{equation}
  Thus (1) holds exactly when there is $\bm\theta\in\bigoplus_iH_i^\perp$ with $\Delta(\bm\theta)\in\bQ_{>0}^N$. Multiplying all $\bm\theta_i$ by a common positive denominator makes their coordinates integers and every coordinate of $\bm\theta_i-\bm\theta_{i+1}$ a positive integer, hence at least $1$. The corresponding maps therefore satisfy (2) with integer coefficients. Conversely, the maps in (2) give $\Delta(\bm\theta)\in\bQ_{>0}^N$, so (1) follows from \eqref{eq:orthant-alternative}.
\end{proof}

\red{The homomorphisms we seek are specified by a possibly infinite system of linear equations and strict inequalities with rational coefficients, each involving only finitely many unknowns. The next lemma shows that rational solvability of every finite subsystem gives one solution of the entire system in a common totally ordered field containing $\bQ$. The common ordered field need not be $\bQ$.}

\begingroup
Recall that an \emph{ultrafilter} on a nonempty set $X$ is a family $\mathcal U$ of subsets of $X$ such that $X\in\mathcal U$, $\varnothing\notin\mathcal U$, and
\begin{equation}
  \begin{aligned}
    A,B\in\mathcal U&\ \Longrightarrow\ A\cap B\in\mathcal U,\\
    A\in\mathcal U,\quad A\subseteq B\subseteq X &\ \Longrightarrow\ B\in\mathcal U,\\
    A\subseteq X&\ \Longrightarrow\ \text{exactly one of }A,\ X\setminus A\text{ belongs to }\mathcal U.
  \end{aligned}
\end{equation}
For example, for any $x\in X$, the family $\{A\subseteq X:x\in A\}$ is an ultrafilter on $X$. The ultrafilter lemma (see, e.g., \cite[Lemma~7.2(iii) and Theorem~7.5, pp.~74--75]{Jech2003}), a consequence of Zorn's lemma, states that any family of subsets of $X$ whose finite intersections are nonempty is contained in an ultrafilter on $X$.
\par\endgroup

\begin{lemma}\label{lem:rational-system-compactness}
  Let $T$ be a set and $\mathscr C$ a set of linear equations and strict inequalities with rational coefficients, each involving finitely many unknowns $u_t$ $(t\in T)$. If every finite subset of $\mathscr C$ has a rational solution, then $\mathscr C$ has a solution $u_t\in\bF$ in some totally ordered field $\bF\supseteq\bQ$.
\end{lemma}

\begin{proof}
  We specialize the ultraproduct proof of the compactness theorem for first-order logic (see, e.g., \cite[Chapter~5, proof of Theorem~4.1, p.~102]{BellSlomson1974}). Let $\cJ=\{J\subseteq\mathscr C:|J|<\infty\}$. For each $J\in\cJ$, choose a rational solution $(a_t^{(J)})_{t\in T}$ of $J$, setting unused coordinates to zero. The sets $B_c=\{J\in\cJ:c\in J\}$ $(c\in\mathscr C)$ have the finite intersection property, since $\{c_1,\ldots,c_r\}\in B_{c_1}\cap\cdots\cap B_{c_r}$. By the ultrafilter lemma, choose an ultrafilter $\mathcal U$ on $\cJ$ containing every $B_c$.

  Form the ultrapower $\bF=\bQ^\cJ/{\sim_{\mathcal U}}$ of the ordered field $\bQ$, where, for $\mathbf q=(q_J)_{J\in\cJ}$ and $\mathbf r=(r_J)_{J\in\cJ}$,
  \begin{equation}\label{eq:ultrapower-equivalence}
    \mathbf q\sim_{\mathcal U}\mathbf r \quad\Longleftrightarrow\quad \{J\in\cJ:q_J=r_J\}\in\mathcal U.
  \end{equation}
  Write $[\mathbf q]_{\mathcal U}$ for the class of $\mathbf q$. By \L o\'s's theorem (see, e.g., \cite[Theorem~12.3 and Corollary~12.5, pp.~159--161]{Jech2003}), $\bF$ is a totally ordered field containing $\bQ$ via \mbox{$q\mapsto[(q)_{J\in\cJ}]_{\mathcal U}$}. Put
  \begin{equation}\label{eq:ultrapower-solution}
    u_t=\bigl[(a_t^{(J)})_{J\in\cJ}\bigr]_{\mathcal U}\in\bF \qquad(t\in T).
  \end{equation}
  For each $c\in\mathscr C$, the chosen solution satisfies $c$ for every $J\in B_c$. Since $B_c\in\mathcal U$, \L o\'s's theorem shows that $(u_t)_{t\in T}$ satisfies $c$ in $\bF$.
\end{proof}

We next solve the equations $\theta_\varphi(H_\varphi)=0$ and the inequalities $\theta_\varphi([S])>\theta_\psi([S])$ for $\varphi<\psi$ simultaneously. Proposition~\ref{lem:finite-dimensional-separation} gives a rational solution for every finite subsystem. Lemma~\ref{lem:rational-system-compactness} then gives a solution of the entire system in one totally ordered field.

\begin{lemma}\label{lem:positive-class-homomorphisms}
  Suppose \red{condition~(2) of Theorem~\ref{thm:positive-class-criterion} holds}. Then there are a totally ordered field $\bF$ and group homomorphisms \red{$\theta_\varphi:K_0(\cA)\to\bF$ $(\varphi\in\Phi)$} such that \red{$\theta_\varphi(H_\varphi)=0$ for $\varphi\in\Phi$}, and \red{$\theta_\varphi([S])>\theta_\psi([S])$ for $\varphi<\psi$} and $S\in\operatorname{Sim}\cA$.
\end{lemma}

\begin{proof}
  Consider scalar unknowns \red{$u_{\varphi,S}$} and the system
  \begin{equation}\label{eq:slice-homomorphism-system}
    \left\{\begin{aligned}
      \sum_Sm_S(E)\red{u_{\varphi,S}}&=0 && \red{(\varphi\in\Phi,\ E\in\Pi_\varphi),}\\
      \red{u_{\varphi,S}}&>\red{u_{\psi,S}} && \red{(\varphi<\psi,\ S\in\operatorname{Sim}\cA).}
    \end{aligned}\right.
  \end{equation}
  \begin{claim}
    Every finite subsystem of \eqref{eq:slice-homomorphism-system} has a rational solution.
  \end{claim}

  \begin{proof}[Proof of the claim]
    Let $J$ be a finite set of equations and strict inequalities chosen from \eqref{eq:slice-homomorphism-system}. \red{We may assume $J$ contains a variable. Put}
    \begingroup
    \begin{equation}
      \begin{aligned}
        \Phi_J&=\{\varphi\in\Phi:\exists S,\ u_{\varphi,S}\text{ occurs in }J\}
        =\{\varphi_1<\cdots<\varphi_{\revision{n}}\},\\
        \Sigma&=\{S\in\operatorname{Sim}\cA:\exists\varphi,\ u_{\varphi,S}\text{ occurs in }J\}.
      \end{aligned}
    \end{equation}
    \endgroup
    Then $\Sigma$ is finite. For each \red{$1\leq i\leq\revision{n}$}, let $H_i^J$ be the $\bQ$-span of the $[E]$ whose equations \red{at phase $\varphi_i$} occur in $J$. Then $H_i^J\subseteq\red{H_{\varphi_i}}\cap\bigoplus_{S\in\Sigma}\bQ[S]$. Identify $\bigoplus_{S\in\Sigma}\bQ[S]$ with $\bQ^{|\Sigma|}$ using the simple classes as a coordinate basis. Since $H_i^J\subseteq\red{H_{\varphi_i}}$, the hypothesis implies condition~(1) of Proposition~\ref{lem:finite-dimensional-separation} for these subspaces. That proposition therefore gives $\bQ$-linear maps $\psi_i:\bigoplus_{S\in\Sigma}\bQ[S]\to\bQ$ with $\psi_i(H_i^J)=0$ and $\psi_i([S])-\psi_{i+1}([S])\geq1$ for \red{$1\leq i<\revision{n}$} and $S\in\Sigma$. \red{For $i<j$ and $S\in\Sigma$,}
    \begingroup
    \begin{equation}
      \psi_i([S])-\psi_j([S])
      =\sum_{k=i}^{j-1}\bigl(\psi_k([S])-\psi_{k+1}([S])\bigr)
      \geq j-i>0.
    \end{equation}
    \endgroup
    Setting \red{$u_{\varphi_i,S}=\psi_i([S])$ for $1\leq i\leq\revision{n}$, $S\in\Sigma$, and all other $u_{\varphi,S}=0$} gives a rational solution of $J$.
  \end{proof}

  By the claim, Lemma~\ref{lem:rational-system-compactness}, applied with $T=\red{\Phi}\times\operatorname{Sim}\cA$, now gives a totally ordered field $\bF\supseteq\bQ$ and elements \red{$u_{\varphi,S}\in\bF$} satisfying all of \eqref{eq:slice-homomorphism-system}. Define
  \begin{equation}\label{eq:ultrapower-slice-homomorphisms}
    \red{\theta_\varphi([S])=u_{\varphi,S}}\qquad\red{(\varphi\in\Phi,\ S\in\operatorname{Sim}\cA).}
  \end{equation}
  Since the $[S]$ form a $\bZ$-basis of $K_0(\cA)$, these values extend uniquely to group homomorphisms \red{$\theta_\varphi:K_0(\cA)\to\bF$}. For \red{$E\in\Pi_\varphi$}, we have \red{$\theta_\varphi(E)=\sum_Sm_S(E)u_{\varphi,S}=0$}, so \red{$\theta_\varphi(H_\varphi)=0$}. Also \red{$\theta_\varphi([S])=u_{\varphi,S}>u_{\psi,S}=\theta_\psi([S])$ for $\varphi<\psi$} and $S\in\operatorname{Sim}\cA$.
\end{proof}

\red{We next consider $|\Phi|<\infty$.}

\begin{lemma}\label{lem:decreasing-cuts-induction}
  \red{Assume $\Phi=\{\varphi_1<\cdots<\varphi_n\}$ and put $H_i=H_{\varphi_i}$.} Let $\bF$ be a totally ordered field, and suppose group homomorphisms $\theta_1,\ldots,\theta_n:K_0(\cA)\to\bF$ satisfy $\theta_i(H_i)=0$ for $1\leq i\leq n$ and $\theta_i([S])>\theta_{i+1}([S])$ for $1\leq i<n$ and $S\in\operatorname{Sim}\cA$. Then there is a stability function $\mu:\operatorname{Obj}^*(\cA)\to[1,n]_\bF$ inducing $(\Phi,\{\Pi_\varphi\}_{\varphi\in\Phi})$, with $\mu(E)=i$ for $1\leq i\leq n$ and $0\ne E\in\Pi_{\varphi_i}$.
\end{lemma}

\begin{proof}
  For $n=1$, take $\mu\equiv1$ and apply Lemma~\ref{lem:HN-recognition}. Assume $n>1$. Set $I=[1,n]_\bF$ and define, for $t\in I$,
  \begin{equation}\label{eq:decreasing-form-interpolation}
    F_t=(i+1-t)\theta_i+(t-i)\theta_{i+1} \qquad(i\leq t\leq i+1,\ 1\leq i<n).
  \end{equation}
  \begin{claim}
    For every nonzero $M\in\cA$, the function $t\mapsto F_t(M)$ is strictly decreasing on $I$, and $F_1(M)\geq0\geq F_n(M)$.
  \end{claim}

  \begin{proof}[Proof of the claim]
    For $S\in\operatorname{Sim}\cA$, choose $j$ with $S\in\Pi_{\varphi_j}$. Then $\theta_1(S)\geq\theta_j(S)=0\geq\theta_n(S)$. For $M\ne0$, choose a composition series $0=M_0\subsetneq M_1\subsetneq\cdots\subsetneq M_r=M$ and put $S_a=M_a/M_{a-1}$. Then $r>0$ and $[M]=\sum_{a=1}^r[S_a]$, so $\theta_1(M)\geq0\geq\theta_n(M)$ and
    \begin{equation}
      \theta_i(M)-\theta_{i+1}(M) =\sum_{a=1}^r\bigl(\theta_i(S_a)-\theta_{i+1}(S_a)\bigr)>0 \quad(1\leq i<n).
    \end{equation}
    Therefore $t\mapsto F_t(M)$ is strictly decreasing when $M$ is fixed.
  \end{proof}

  For $M\ne0$, if $\theta_1(M)=0$, set $\mu(M)=1$. Otherwise, the claim ensures that there is a unique integer $1\leq i<n$ with $\theta_i(M)>0\geq\theta_{i+1}(M)$; define
  \begin{equation}\label{eq:small-phase-uniform-function}
    \mu(M)=i+\frac{\theta_i(M)}{\theta_i(M)-\theta_{i+1}(M)}.
  \end{equation}
  In either case, $\mu(M)\in I$ and $F_{\mu(M)}(M)=0$. By the claim, this zero is unique and the sign conditions of Proposition~\ref{prop:red-path-criterion} hold. Hence $\mu$ is a stability function. Finally, $F_i(E)=\theta_i(E)=0$ for $0\ne E\in\Pi_{\varphi_i}$, hence $\mu(E)=i$. Lemma~\ref{lem:HN-recognition} shows that $\mu$ induces $(\Phi,\{\Pi_\varphi\}_{\varphi\in\Phi})$.
\end{proof}

\begingroup
\begin{lemma}\label{lem:finite-chain-local-signing}
  Suppose condition~(2) of Theorem~\ref{thm:positive-class-criterion} holds. Then every finite subset $\mathscr F\subseteq\mathscr E$ admits an admissible signing.
\end{lemma}

\begin{proof}
  If $\mathscr F=\varnothing$, its unique signing is admissible. Assume $\mathscr F\ne\varnothing$, and let $V_{\mathscr F}$ be the finite set of nonzero objects occurring in its triples, represented in the fixed skeleton as in Definition~\ref{def:signing}. Put
  \begin{equation}
    \Phi_{\mathscr F}=\{\varphi\in\Phi:V_{\mathscr F}\cap\Pi_\varphi\ne\varnothing\}.
  \end{equation}

  If $\Phi_{\mathscr F}=\varnothing$, assign sign $0$ to every triple in $\mathscr F$. This signing is admissible trivially. Assume now that $\Phi_{\mathscr F}\ne\varnothing$, and write $\Phi_{\mathscr F}=\{\varphi_1<\cdots<\varphi_{\revision{n}}\}$, where $\revision{n}\geq1$. Define
  \begin{equation}
    \Sigma=\bigcup_{X\in V_{\mathscr F}}\operatorname{supp}([X]),\qquad
    W_\Sigma=\bigoplus_{S\in\Sigma}\bQ[S],\qquad
    W_{\Sigma,+}=\bigoplus_{S\in\Sigma}\bQ_{\geq0}[S].
  \end{equation}
  Since $V_{\mathscr F}$ is finite and every $[X]$ has finite support, $\Sigma$ is finite. For $1\leq i\leq\revision{n}$, put
  \begin{equation}
    L_i=\operatorname{span}_\bQ\{[E]:E\in V_{\mathscr F}\cap\Pi_{\varphi_i}\}
    \subseteq H_{\varphi_i}\cap W_\Sigma.
  \end{equation}
  The inclusions $L_i\subseteq H_{\varphi_i}$ and the hypothesis give condition~(1) of Proposition~\ref{lem:finite-dimensional-separation} for these subspaces of $W_\Sigma$. Hence there are $t_i\in\operatorname{Hom}_\bQ(W_\Sigma,\bQ)$ with $t_i(L_i)=0$ and $t_i([S])>t_{i+1}([S])$ for $1\leq i<\revision{n}$, $S\in\Sigma$.

  Choose $c\in\bQ_{>0}$ such that $c>|t_i([S])|$ for $1\leq i\leq\revision{n}$, $S\in\Sigma$, and put $t_0=c\ell|_{W_\Sigma}$ and $t_{\revision{n+1}}=-c\ell|_{W_\Sigma}$. Thus $t_i([S])>t_{i+1}([S])$ for $0\leq i\leq\revision{n}$, $S\in\Sigma$. Define
  \begin{equation}\label{eq:local-signing-interpolation}
    F_s=(i+1-s)t_i+(s-i)t_{i+1}
    \quad(i\leq s\leq i+1,\ 0\leq i\leq\revision{n}),
    \qquad s\in[0,\revision{n+1}]_\bQ.
  \end{equation}
  For $\mathbf v\in W_{\Sigma,+}\setminus\{\mathbf0\}$, we have $t_0(\mathbf v)>0>t_{\revision{n+1}}(\mathbf v)$ and $t_i(\mathbf v)>t_{i+1}(\mathbf v)$. Hence there is a unique $0\leq i\leq\revision{n}$ with $t_i(\mathbf v)>0\geq t_{i+1}(\mathbf v)$. Set
  \begin{equation}\label{eq:local-signing-root}
    r(\mathbf v)=i+\frac{t_i(\mathbf v)}{t_i(\mathbf v)-t_{i+1}(\mathbf v)}.
  \end{equation}
  Then $F_{r(\mathbf v)}(\mathbf v)=0$ and $F_s(\mathbf v)$ is strictly decreasing in $s$. \revision{For $\mathbf v,\mathbf w\in W_{\Sigma,+}\setminus\{\mathbf0\}$, put $s=r(\mathbf v+\mathbf w)$. Additivity gives
  \begin{equation}
    F_s(\mathbf v)+F_s(\mathbf w)=F_s(\mathbf v+\mathbf w)=0.
  \end{equation}
  The strict sign change at each root gives strict see-saw under addition, by the same argument as in Proposition~\ref{prop:red-path-criterion}.}

  For $(A,B,C)\in\mathscr F$, let $\varepsilon(A,B,C)\in\{+,0,-\}$ be the sign of $r([C])-r([A])$. If $E\in V_{\mathscr F}\cap\Pi_{\varphi_i}$, then $t_i([E])=0$, so $r([E])=i$. Therefore $X\mapsto r([X])$ descends to $V_{\mathscr F}/\sim_\varepsilon$ and strictly increases along each arrow of $Q_{\mathscr F,\varepsilon}$. Hence $\varepsilon$ is admissible.
\end{proof}
\par\endgroup

\begin{proof}[Proof of Theorem~\ref{thm:positive-class-criterion}]
  Lemma~\ref{lem:positive-class-obstruction} proves (1)$\Rightarrow$(2), and Lemma~\ref{lem:positive-class-homomorphisms} proves (2)$\Rightarrow$(3). \red{Lemma~\ref{lem:finite-chain-local-signing} and Proposition~\ref{prop:reduction to finite case} prove (2)$\Rightarrow$(1).}

  \begingroup
  For (3)$\Rightarrow$(2), fix a finite chain $\varphi_1<\cdots<\varphi_n$ and put $H_i=H_{\varphi_i}$, $\theta_i=\theta_{\varphi_i}$. If $\mathbf x_0,\ldots,\mathbf x_n$ satisfy \eqref{eq:positive-class-flow}, then
  \begin{equation}\label{eq:finite-chain-separator-duality}
    \sum_{i=1}^{n-1}(\theta_i-\theta_{i+1})(\mathbf x_i)
    =\sum_{i=1}^{n}\theta_i(\mathbf x_i-\mathbf x_{i-1})=0.
  \end{equation}
  Since $(\theta_i-\theta_{i+1})(\mathbf x)>0$ for $\mathbf x\in K_0(\cA)_{\bQ_{\geq0}}\setminus\{\mathbf0\}$, we obtain $\mathbf x_0=\cdots=\mathbf x_n=\mathbf0$.
  \par\endgroup

\end{proof}

\begingroup
\begin{corollary}\label{cor:finite-phase-integral-separators}
  Suppose the equivalent conditions of Theorem~\ref{thm:positive-class-criterion} hold, $\Phi=\{\varphi_1<\cdots<\varphi_n\}$ is finite, and $K_0(\cA)$ has finite rank. Then the homomorphisms in condition~(3) can be chosen as $\theta_i:=\theta_{\varphi_i}:K_0(\cA)\to\bZ$, with
  \begin{equation}
    \theta_i([S])-\theta_{i+1}([S])\geq1
    \qquad(S\in\operatorname{Sim}\cA,\ 1\leq i<n).
  \end{equation}
  Moreover, there is a stability function $\mu:\operatorname{Obj}^*(\cA)\to\bQ$ \revision{inducing $(\Phi,\{\Pi_\varphi\}_{\varphi\in\Phi})$.}
\end{corollary}

\begin{proof}
  Put $d=\operatorname{rank}K_0(\cA)$ and $H_i=H_{\varphi_i}$. Identify the simple classes with the standard basis of $\bZ^d$. Condition~(2) of Theorem~\ref{thm:positive-class-criterion}, applied to the full phase chain, and Proposition~\ref{lem:finite-dimensional-separation} give the integer-valued homomorphisms with the stated inequalities. Lemma~\ref{lem:decreasing-cuts-induction}, applied with $\bF=\bQ$, then gives the required stability function.
\end{proof}
\par\endgroup

\begingroup
\begin{remark}\label{rem:finite-chain-criterion}
  We note that Theorem~\ref{thm:positive-class-criterion} does not assume that the semistable subcategories are wide.

  If $\Phi=\{\varphi_1<\cdots<\varphi_{\revision{n}}\}$ is finite, it suffices to check condition~(2) on the full phase chain: any solution for a subchain extends to one for the full chain by inserting zero increments. Thus the condition reduces to the system \eqref{eq:positive-class-flow} \revision{for the full phase chain, with} $H_i=H_{\varphi_i}$. If $K_0(\cA)$ also has finite rank, Corollary~\ref{cor:finite-phase-integral-separators} gives integer-valued homomorphisms satisfying condition~(3) and a rational-valued inducing stability function. \revision{Finiteness of $\Phi$ and finite rank of $K_0(\cA)$ are not required for the equivalence in Theorem~\ref{thm:positive-class-criterion}.}

  In general, \revision{condition~(2) shows that $(\Phi,\{\Pi_\varphi\}_{\varphi\in\Phi})$ is not induced by any stability function} if and only if there exist a finite chain $\varphi_1<\cdots<\varphi_n$ in $\Phi$, finite sets $\mathcal T_i\subseteq\Pi_{\varphi_i}$ and $\Sigma\subseteq\operatorname{Sim}\cA$ with $\operatorname{supp}([E])\subseteq\Sigma$ for $E\in\bigcup_i\mathcal T_i$, and $\mathbf x_0,\ldots,\mathbf x_n$ satisfying
  \begin{equation}
    \left\{\begin{aligned}
      \mathbf x_0=\mathbf x_n&=\mathbf0,\qquad (\mathbf x_0,\ldots,\mathbf x_n)\ne(\mathbf0,\ldots,\mathbf0),\\
      \mathbf x_i&\in\bigoplus_{S\in\Sigma}\bQ_{\geq0}[S] &&(1\leq i<n),\\
      \mathbf x_i-\mathbf x_{i-1}&\in\operatorname{span}_\bQ\{[E]:E\in\mathcal T_i\} &&(1\leq i\leq n).
    \end{aligned}\right.
  \end{equation}
\end{remark}
\par\endgroup

\section{Wide semistable subcategories and phase bounds}\label{sec:wide-phase-bounds}

\red{Throughout this section, fix $\Phi=\{\varphi_1<\cdots<\varphi_n\}$ and put $H_i=H_{\varphi_i}$.} We apply Theorem~\ref{thm:positive-class-criterion} to stability data whose semistable subcategories are wide. For at most six phases, we construct homomorphisms satisfying condition~(3). A seven-phase example shows that wideness alone does not ensure the existence of an inducing stability function. We retain the notation in \eqref{eq:phase-class-spaces}.

\revision{The six-phase bound comes from three parts of the construction. Lemma~\ref{lem:adjacent-numerical-cuts} handles at most two central phases simultaneously, and Lemma~\ref{lem:two-fringe-extension} extends the resulting homomorphisms across at most two phases on either side. Thus the proof covers $2+2+2=6$ phases. The construction uses $-f(X)\geq1$ whenever $f(X)<0$, and $g(X)\geq1$ whenever $g(X)>0$. These inequalities follow from the integer values and remain valid when there are infinitely many ambient simples.}

\subsection{Homomorphisms for adjacent phases}

Throughout this subsection, fix a stability data $(\Phi,\{\Pi_\varphi\}_{\varphi\in\Phi})$ on $\cA$, where $\Phi=\{\varphi_1<\cdots<\varphi_n\}$ and every $\Pi_\varphi$ is wide.

The semistable subcategories at the extreme phases are Serre, which will control subobjects and quotients near the boundary phases.

\begin{lemma}\label{lem:extreme slices are Serre}
  Let $(\Phi,\{\Pi_\varphi\}_{\varphi\in\Phi})$ be a system of stability data on $\cA$, and suppose that every $\Pi_\varphi$ is wide. If $\alpha$ is the minimal element of $\Phi$, then $\Pi_\alpha$ is Serre. Dually, if $\beta$ is the maximal element of $\Phi$, then $\Pi_\beta$ is Serre.
\end{lemma}

\begin{proof}
  Let $\alpha$ be minimal and let $X\in\Pi_\alpha$. For $0\ne U\subsetneq X$, Lemma~\ref{lem:HN-phase-bounds}(1) gives $\alpha\leq\phi^-(U)\leq\phi^+(U)\leq\alpha$, so $U\in\Pi_\alpha$. The cases $U=0$ and $U=X$ are immediate. Since $\Pi_\alpha$ is wide, the cokernel $X/U$ also belongs to $\Pi_\alpha$. Thus $\Pi_\alpha$ is Serre. The assertion for a maximal phase is dual.
\end{proof}

We use this property to bound the lengths of simple objects in the semistable subcategories at the second and penultimate phases.

\begin{lemma}\label{lem:fringe-simple-quotients}
  Let $(\Phi,\{\Pi_\varphi\}_{\varphi\in\Phi})$ be a stability data on $\cA$, where $\Phi=\{\varphi_1<\cdots<\varphi_n\}$ and $n\geq3$. Assume that every $\Pi_\varphi$ is wide.
  \begin{enumerate}
    \item\label{item:fringe-simple-quotients} If $T$ is simple in $\Pi_{\varphi_2}$ but not ambient simple, then there is a nonsplit exact sequence
    \begin{equation}\label{eq:fringe-simple-head}
      0\longrightarrow K\longrightarrow T\longrightarrow S \longrightarrow0,
    \end{equation}
    where $K\in\Pi_{\varphi_1}$ and $S$ is ambient simple in $\Pi_{\varphi_j}$ for some $2<j\leq n$. Every nonzero quotient of $T$ in $\cA$ is simple in a semistable subcategory. In particular, $\ell(T)\leq j-1$.
    \item\label{item:fringe-simple-subobjects} Dually, if $T$ is simple in $\Pi_{\varphi_{n-1}}$ but not ambient simple, then there is a nonsplit exact sequence
    \begin{equation}
      0\longrightarrow L\longrightarrow T\longrightarrow Q \longrightarrow0,
    \end{equation}
    where $Q\in\Pi_{\varphi_n}$ and $L$ is ambient simple in $\Pi_{\varphi_j}$ for some $1\leq j<n-1$. Every nonzero subobject of $T$ in $\cA$ is simple in a semistable subcategory. In particular, $\ell(T)\leq n-j$.
  \end{enumerate}
\end{lemma}

\begin{proof}
  We prove (\ref{item:fringe-simple-quotients}); (\ref{item:fringe-simple-subobjects}) follows by passing to $\cA^{\mathrm{op}}$ and reversing the phase order.

  Choose an ambient simple quotient $p:T\to S$, with $S\in\Pi_{\varphi_j}$, and put $K=\ker p$. Since $T$ is not ambient simple, $K\ne0$. Lemma~\ref{lem:simple-slice-phase-bounds}(\ref{item:simple-slice-subobjects})--(\ref{item:simple-slice-quotients}) gives $j>2$ and $K\in\Pi_{\varphi_1}$. The sequence \eqref{eq:fringe-simple-head} is nonsplit, since a section would give a nonzero map $S\to T$, contrary to Hom-vanishing.

  Every proper subobject $u:U\hookrightarrow T$ lies in $\Pi_{\varphi_1}$ by Lemma~\ref{lem:simple-slice-phase-bounds}(\ref{item:simple-slice-subobjects}). This semistable subcategory is Serre by Lemma~\ref{lem:extreme slices are Serre}, so $\im(pu)\in\Pi_{\varphi_1}$. Since $S$ is ambient simple and $S\notin\Pi_{\varphi_1}$, we have $pu=0$. Thus $u$ factors through the kernel $K\hookrightarrow T$. Hence $K$ is the unique maximal subobject of $T$ in $\cA$, and $S$ is its simple top in $\cA$.

  \begin{claim}
    Every nonzero quotient $Q$ of $T$ is semistable and simple in its semistable subcategory. Moreover, every proper subobject of $Q$ belongs to $\Pi_{\varphi_1}$.
  \end{claim}

  \begin{proof}[Proof of the claim]
    Let $q:T\twoheadrightarrow Q$ be a nonzero quotient and $v:V\hookrightarrow Q$ a proper subobject. Form the pullback $U=T\times_Q V$ and put $N=\ker q$. This gives a commutative diagram with exact rows, whose right-hand square is a pullback:
    \begin{equation}
      \begin{tikzcd}[column sep=large,row sep=large]
        0\arrow[r]&N\arrow[r]\arrow[d,equal] &U\arrow[r,two heads]\arrow[d,hook] &V\arrow[r]\arrow[d,hook,"v"]&0\\
        0\arrow[r]&N\arrow[r] &T\arrow[r,two heads,"q"]&Q\arrow[r]&0.
      \end{tikzcd}
    \end{equation}
    The induced monomorphism $U\hookrightarrow T$ is proper, so $U\in\Pi_{\varphi_1}$. Since $U\twoheadrightarrow V$ is epic and $\Pi_{\varphi_1}$ is Serre, we also have $V\in\Pi_{\varphi_1}$.

    Let $0\ne F\subseteq Q$ be its maximal semistable subobject. Since $Q$ is a nonzero quotient of $T$, Lemma~\ref{lem:HN-phase-bounds}(1) gives $\phi^-(Q)\geq\phi(T)=\varphi_2$ (the isomorphism case is immediate). The phase of $F$ is therefore
    \begin{equation}
      \phi(F)=\phi^+(Q)\geq\phi^-(Q)\geq\varphi_2>\varphi_1.
    \end{equation}
    If $F\subsetneq Q$, the preceding pullback argument, applied with $V=F$, gives $F\in\Pi_{\varphi_1}$, hence $\phi(F)=\varphi_1$, contradicting $\phi(F)>\varphi_1$. Thus $F=Q$. Every proper subobject of $Q$ lies in $\Pi_{\varphi_1}$, whereas $Q$ has phase greater than $\varphi_1$; hence $Q$ is simple in its semistable subcategory.
  \end{proof}

  Choose a composition series $0=K_0\subset K_1\subset\cdots\subset K_m=K$ in $\cA$. Since $K_a\subseteq K\subsetneq T$, each $T/K_a$ is a nonzero quotient of $T$. According to the claim, it is semistable and simple in its semistable subcategory. Write $\phi(T/K_a)=\varphi_{i_a}$ for $0\leq a\leq m$. For $0\leq a<m$, the object $T/K_{a+1}$ is a proper nonzero quotient of $T/K_a$, so Lemma~\ref{lem:simple-slice-phase-bounds}(\ref{item:simple-slice-quotients}) gives $i_a<i_{a+1}$. As $K_0=0$ and $T/K_m\simeq S$, we obtain
  \begin{equation}
    2=i_0<i_1<\cdots<i_m=j.
  \end{equation}
  Therefore $\ell(T)=m+1\leq j-1$.
\end{proof}

The following uniqueness property will make later prescriptions on ambient simple objects well-defined.

\begin{lemma}\label{lem:fringe-head-uniqueness}
  Let $(\Phi,\{\Pi_\varphi\}_{\varphi\in\Phi})$ be a stability data on $\cA$, where $\Phi=\{\varphi_1<\cdots<\varphi_n\}$, $n\geq3$, and every $\Pi_\varphi$ is wide. Then nonisomorphic objects that are simple in $\Pi_{\varphi_2}$ but not in $\cA$ have nonisomorphic simple tops in $\cA$. Dually, nonisomorphic objects that are simple in $\Pi_{\varphi_{n-1}}$ but not in $\cA$ have nonisomorphic simple socles in $\cA$.
\end{lemma}

\begin{proof}
  Suppose $T,T'$ have isomorphic simple tops in $\cA$, identified with $S$, and write $p:T\twoheadrightarrow S$ and $p':T'\twoheadrightarrow S$ for the quotient maps. Set $K=\ker p$ and $K'=\ker p'$; both lie in $\Pi_{\varphi_1}$. Form the pullback $P=T\times_S T'$, with projections $\pi:P\to T$ and $\pi':P\to T'$. Then $\ker\pi\simeq K'$ and $\ker\pi'\simeq K$. Applying Lemma~\ref{lem:HN-phase-bounds}(2) to $0\to K'\to P\to T\to0$ gives $\phi^+(P)\leq\varphi_2$. Moreover, $P\notin\Pi_{\varphi_1}$: otherwise its quotient $T$ would lie in this Serre category. Therefore the maximal semistable subobject $\iota:P_1\hookrightarrow P$ of $P$ in $\cA$ lies in $\Pi_{\varphi_2}$. The composite $f=\pi\iota:P_1\to T$ is nonzero: otherwise $\iota$ would factor through $\ker\pi\simeq K'\in\Pi_{\varphi_1}$, contrary to Hom-vanishing. Since $T$ is simple in $\Pi_{\varphi_2}$, $f$ is an epimorphism in $\Pi_{\varphi_2}$.

  \revision{Let $J=\ker f$. Wideness gives $J\in\Pi_{\varphi_2}$. The monomorphism $J\hookrightarrow P_1\hookrightarrow P$ factors through $\ker\pi\simeq K'$, since its composite with $\pi$ is zero. Thus $J\hookrightarrow K'$. Hom-vanishing from $\Pi_{\varphi_2}$ to $\Pi_{\varphi_1}$ forces $J=0$. Hence $f$ is an isomorphism and $P_1\simeq T$. Interchanging $T$ and $T'$ gives $P_1\simeq T'$. The socle assertion holds dually.}
\end{proof}

To construct the homomorphisms $f,g$ in Lemma~\ref{lem:adjacent-numerical-cuts}, we will use HRS hearts. The next two lemmas establish the required wideness and length properties.

\begin{lemma}\label{lem:abelian-slices-wide}
  Let $\mathcal H$ be an essentially small abelian category with a stability data $(\Phi,\{\Pi_\varphi\}_{\varphi\in\Phi})$. For every $\varphi\in\Phi$, if $\Pi_\varphi$ is an abelian category, then it is wide in $\mathcal H$.
\end{lemma}

\begin{proof}
  Fix an abelian semistable subcategory $\Pi_\varphi$ and a morphism $f:X\to Y$ in $\Pi_\varphi$. By the definition of stability data, $\Pi_\varphi$ is extension-closed. It remains to verify that
  \begin{equation}
    \ker_{\Pi_\varphi}f\cong\ker_{\mathcal H}f, \qquad \coker_{\Pi_\varphi}f \cong\coker_{\mathcal H}f.
  \end{equation}
  All unmarked kernels, images and quotients below are computed in $\mathcal H$. Put $I=\im f$. If $I=0$, the assertion is immediate. Otherwise, Lemma~\ref{lem:HN-phase-bounds}(1) applied to $X\twoheadrightarrow I\hookrightarrow Y$ gives $\varphi\leq\phi^-(I)\leq\phi^+(I)\leq\varphi$, so $I\in\Pi_\varphi$.

  Put $K=\ker f$. By Lemma~\ref{lem:HN-phase-bounds}(1), every HN phase of $K$ is at most $\varphi$. Let $K_\varphi$ be the maximal semistable subobject of $K$ in $\mathcal H$ if phase $\varphi$ occurs, and put $K_\varphi=0$ otherwise. Then $K_\varphi\in\Pi_\varphi$, and every HN phase of $K_{<\varphi}:=K/K_\varphi$ is strictly below $\varphi$. Via the composite $K_\varphi\to K\to X$, regard $K_\varphi$ as a subobject of $X$ and put $E=X/K_\varphi$. The canonical epimorphism $X\to I$ vanishes on $K_\varphi$, so it induces an epimorphism $p:E\to I$. The inclusion $K\to X$ induces a monomorphism $K_{<\varphi}\to E$ whose image is $\ker_{\mathcal H}p$. Thus the inclusions $K_\varphi\subseteq K\subseteq X$ induce the following commutative diagram in $\mathcal H$, with exact rows and columns:
  \begin{equation}
    \begin{tikzcd}[column sep=2.8em,row sep=2.1em]
      & 0 \arrow[d] & 0 \arrow[d] & \\
      & K_\varphi \arrow[r,equal] \arrow[d] & K_\varphi \arrow[d] & \\
      0 \arrow[r] & K \arrow[r] \arrow[d] & X \arrow[r] \arrow[d] & I \arrow[r] \arrow[d,equal] & 0 \\
      0 \arrow[r] & K_{<\varphi} \arrow[r] \arrow[d] & E \arrow[r,"p"] \arrow[d] & I \arrow[r] & 0 \\
      & 0 & 0
    \end{tikzcd}
  \end{equation}
  If $K_{<\varphi}=0$, then $E\cong I\in\Pi_\varphi$. Otherwise, since $E$ is a quotient of $X$, Lemma~\ref{lem:HN-phase-bounds}(1)--(2) gives
  \begin{equation}
    \varphi\leq\phi^-(E)\leq\phi^+(E) \leq\max\{\phi^+(K_{<\varphi}),\phi(I)\}=\varphi.
  \end{equation}
  Thus $E\in\Pi_\varphi$ in either case.

  The map $p:E\to I$ is epic in $\Pi_\varphi$, since it is epic in $\mathcal H$ and $\Pi_\varphi$ is full. It is also monic in $\Pi_\varphi$, because $\hom_{\mathcal H} (Z,K_{<\varphi})=0$ for every $Z\in\Pi_\varphi$. As $\Pi_\varphi$ is abelian, $p$ is an isomorphism. Hence $K_{<\varphi}=0$ and $\ker_{\mathcal H}f=K_\varphi\in\Pi_\varphi$. Dually, $\coker_{\mathcal H}f\in\Pi_\varphi$. Fullness shows that these kernels and cokernels also satisfy their universal properties in $\Pi_\varphi$, giving the required identifications. Therefore $\Pi_\varphi$ is wide.
\end{proof}

\begin{lemma}\label{lem:length-from-wide-HN}
  Let $\mathcal H$ be an essentially small abelian category with a stability data $(\Phi,\{\Pi_\varphi\}_{\varphi\in\Phi})$. If $\Phi$ is finite and every $\Pi_\varphi$ is an abelian length category, then $\mathcal H$ is a length category.
\end{lemma}

\begin{proof}
  By Lemma~\ref{lem:abelian-slices-wide}, every semistable subcategory is wide in $\mathcal H$, so its inclusion is exact. Every object of $\mathcal H$ has a finite HN filtration. Each HN factor has a composition series in its own semistable subcategory. The inclusions of the semistable subcategories into $\mathcal H$ are exact. Pulling back these composition series along the HN quotient maps gives a finite filtration in $\mathcal H$ whose factors are simple in their respective semistable subcategories. These factors need not be simple in $\mathcal H$. If they all had finite length in $\mathcal H$, so would the original object, since finite-length objects are closed under extensions. Hence, if $\mathcal H$ is not a length category, some object simple in a semistable subcategory has infinite length in $\mathcal H$. Choose such an $X$ whose phase is minimal.

  \begin{claim}
    Every proper subobject of any quotient of $X$ has finite length in $\mathcal H$.
  \end{claim}

  \begin{proof}[Proof of the claim]
    By Lemma~\ref{lem:simple-slice-phase-bounds}(\ref{item:simple-slice-subobjects}) and the minimality of $\phi(X)$, every proper subobject of $X$ has finite length in $\mathcal H$. For a quotient $X\twoheadrightarrow Q$ and a proper subobject $V\hookrightarrow Q$, form the pullback $P=X\times_Q V$. Then $P\hookrightarrow X$ is a proper monomorphism and $P\twoheadrightarrow V$ is an epimorphism. Thus $P$, and hence $V$, has finite length in $\mathcal H$.
  \end{proof}

  Choose $0\ne U\subsetneq X$ in $\mathcal H$. By the claim, $U$ has finite length, so $X/U$ has infinite length. Its lowest HN factor $Y$ also has infinite length: the kernel of $X/U\to Y$ is a proper subobject of a quotient of $X$ and hence has finite length by the claim. Choose a simple quotient $X'$ of $Y$ in its semistable subcategory. The kernel of $Y\twoheadrightarrow X'$ is a proper subobject of the quotient $Y$ of $X$, so the claim again gives finite length. Thus $X'$ has infinite length in $\mathcal H$. Since $X'$ is a proper nonzero quotient of $X$, Lemma~\ref{lem:simple-slice-phase-bounds}(\ref{item:simple-slice-quotients}) gives $\phi(X')>\phi(X)$. The claim applies to every quotient of $X'$, since each is also a quotient of $X$. \revision{At every subsequent step the new object remains a quotient of the original $X$, so the same claim continues to apply. Iteration therefore produces objects simple in their semistable subcategories with strictly increasing phases, contrary to finiteness of the phase set.} Thus $\mathcal H$ is a length category.
\end{proof}

\revision{Lemmas~\ref{lem:abelian-slices-wide} and~\ref{lem:length-from-wide-HN} allow us to rotate the phase order inside a length heart.}

\begin{lemma}\label{lem:rotated-HN-hearts}
  Let $(\Phi,\{\Pi_\varphi\}_{\varphi\in\Phi})$ be a stability data on an essentially small abelian length category $\cA$, where $\Phi=\{\varphi_1<\cdots<\varphi_n\}$. Assume that every $\Pi_\varphi$ is wide. For each $1\leq i\leq n$, put $\cT_i=\filt(\Pi_{\varphi_i},\ldots,\Pi_{\varphi_n})$ and $\cF_i=\filt(\Pi_{\varphi_1},\ldots,\Pi_{\varphi_{i-1}})$, with $\cF_1=0$. The HRS heart
  \begin{equation}\label{eq:rotated-HN-heart}
    \mathcal H_i=\cF_i[1]*\cT_i \subseteq D^b(\cA)
  \end{equation}
  is a length category. Here $\cF_i[1]*\cT_i$ consists of the objects $E$ admitting a distinguished triangle $F[1]\to E\to T\to F[2]$ with $F\in\cF_i$ and $T\in\cT_i$. It carries a stability data $(\{1<\cdots<n\},\{\Pi^{(i)}_r\}_{r=1}^n)$, where
  \begin{equation}\label{eq:rotated-HN-order}
    (\Pi^{(i)}_1,\ldots,\Pi^{(i)}_n) =(\Pi_{\varphi_i},\ldots,\Pi_{\varphi_n}, \Pi_{\varphi_1}[1],\ldots,\Pi_{\varphi_{i-1}}[1]).
  \end{equation}
  Every $\Pi^{(i)}_r$ is wide in $\mathcal H_i$. When $i=1$, the shifted semistable subcategories in \eqref{eq:rotated-HN-order} are absent. In particular, $\Pi_{\varphi_i}$ and, for $i>1$, $\Pi_{\varphi_{i-1}}[1]$ are Serre subcategories of $\mathcal H_i$.
\end{lemma}

\begin{proof}
  Fix $1\leq i\leq n$. HN filtrations give the torsion pair $(\cT_i,\cF_i)$. The Happel--Reiten--Smal\o{} tilting construction (see \cite[Chapter~I, Section~2]{HRS1996}) shows that $\mathcal H_i$ in \eqref{eq:rotated-HN-heart} is the heart of a bounded $t$-structure on $D^b(\cA)$. Every object $E$ of this heart has a short exact sequence in it
  \begin{equation}
    0\longrightarrow H^{-1}(E)[1]\longrightarrow E \longrightarrow H^0(E)\longrightarrow0,
  \end{equation}
  with $H^{-1}(E)\in\cF_i$ and $H^0(E)\in\cT_i$. \review{The terms and factors of their HN filtrations belong to $\cF_i$ and $\cT_i$, respectively. Each short exact sequence in these filtrations determines a distinguished triangle in $D^b(\cA)$. After shifting the $\cF_i$ part by $[1]$, all three terms of each triangle lie in $\mathcal H_i$, so it is a short exact sequence in this heart.} Combine the shifted HN filtration of $H^{-1}(E)$ with the pullbacks of the HN filtration of $H^0(E)$ along $E\twoheadrightarrow H^0(E)$. The resulting factors occur in decreasing order in the list \eqref{eq:rotated-HN-order}. Within each of the two blocks in that list, \revision{Hom-vanishing follows from the defining condition of $(\Phi,\{\Pi_\varphi\}_{\varphi\in\Phi})$}, and $\hom_{D^b(\cA)}(\cF_i[1],\cT_i)=0$. For $A,B\in\Pi_{\varphi_j}$, the identity $\hom_{D^b(\cA)}(B,A[1])=\ext^1_\cA(B,A)$ identifies extensions of $B$ by $A$ with short exact sequences in $\cA$. Their middle terms lie in $\Pi_{\varphi_j}$. For $j<i$, shifting these sequences by $[1]$ gives the extensions between objects of $\Pi_{\varphi_j}[1]$ in $\mathcal H_i$. Thus the semistable subcategories in \eqref{eq:rotated-HN-order} are extension-closed, and the resulting filtrations define a stability data.

  Each original semistable subcategory is a wide subcategory of the length category $\cA$, hence an abelian length category. The shift functor gives an equivalence with its shifted copy, so every semistable subcategory in \eqref{eq:rotated-HN-order} is also an abelian length category. Lemma~\ref{lem:abelian-slices-wide} makes these semistable subcategories wide in $\mathcal H_i$, and Lemma~\ref{lem:length-from-wide-HN} then shows that $\mathcal H_i$ is a length category. The final assertion follows from Lemma~\ref{lem:extreme slices are Serre}.
\end{proof}

Under the hypotheses of Lemma~\ref{lem:rotated-HN-hearts}, fix $1\leq i\leq n$. The inclusion of the bounded heart $\mathcal H_i$ into $D^b(\cA)$ induces an isomorphism
\begin{equation}\label{eq:heart-class-map}
  \begin{aligned}
    \iota_i:K_0(\mathcal H_i)&\xrightarrow{\sim}K_0(\cA),\\
    [E]_{\mathcal H_i}&\longmapsto [H^0(E)]_\cA-[H^{-1}(E)]_\cA,
  \end{aligned}
\end{equation}
where cohomology is taken with respect to the original heart $\cA$.

\revision{Using simple classes in the HRS hearts $\mathcal H_i$, we construct integer-valued homomorphisms $f,g$ with $f(H_i)=g(H_{i+1})=0$ and $f-g=\ell$.}

\begin{lemma}\label{lem:adjacent-numerical-cuts}
  Let $\cA$ be an essentially small abelian length category with a stability data $(\Phi,\{\Pi_\varphi\}_{\varphi\in\Phi})$, where $\Phi=\{\varphi_1<\cdots<\varphi_n\}$, and assume that every semistable subcategory is wide. For each $1\leq i<n$, there exist group homomorphisms $f,g:K_0(\cA)\to\bZ$ with the following properties.

  For every $1\leq j\leq n$ and every nonzero $X\in\Pi_{\varphi_j}$, their values have the signs
  \begin{equation}\label{eq:adjacent-form-signs}
    \renewcommand{\arraystretch}{1.2}
    \begin{array}{c|cccc}
      &j<i&j=i&j=i+1&j>i+1\\\hline
      f(X)&<0&=0&>0&>0\\
      g(X)&<0&<0&=0&>0
    \end{array}.
  \end{equation}
  Moreover, as homomorphisms on $K_0(\cA)$,
  \begin{equation}\label{eq:adjacent-numerical-cuts}
    f-g=\ell.
  \end{equation}
  Thus $f(M)-g(M)=\ell(M)>0$ for every nonzero $M\in\cA$.
\end{lemma}

\begin{proof}
  Fix $1\leq i<n$ and consider the length heart $\mathcal H_{i+1}$ in Lemma~\ref{lem:rotated-HN-hearts}. The subcategories $\Pi_{\varphi_i}[1]$ and $\Pi_{\varphi_{i+1}}$ are Serre in this heart.

  Let $\mathcal S$ be a set of representatives of the isomorphism classes of simple objects in $\mathcal H_{i+1}$. Then $\{\iota_{i+1}([S]_{\mathcal H_{i+1}}):S\in\mathcal S\}$ is a $\bZ$-basis of $K_0(\cA)$. Define $f:K_0(\cA)\to\bZ$ on this basis by
  \begin{equation}\label{eq:adjacent-f-simple-values}
    f\bigl(\iota_{i+1}([S]_{\mathcal H_{i+1}})\bigr)=
    \begin{cases}
      0,&S\in\Pi_{\varphi_i}[1],\\
      \ell(S),&S\in\Pi_{\varphi_{i+1}},\\
      1+\ell(H^0(S)),&S\notin\Pi_{\varphi_i}[1]\cup\Pi_{\varphi_{i+1}},
    \end{cases}
  \end{equation}
  and put $g=f-\ell$ as a homomorphism on $K_0(\cA)$. Here $S$ is simple in $\mathcal H_{i+1}$; all cohomology objects and composition lengths are taken in $\cA$. Equation~\eqref{eq:heart-class-map} with $i+1$ in place of $i$ gives
  \begin{equation}\label{eq:adjacent-g-simple-values}
    g\bigl(\iota_{i+1}([S]_{\mathcal H_{i+1}})\bigr)=
    \begin{cases}
      \ell(H^{-1}(S)),&S\in\Pi_{\varphi_i}[1],\\
      0,&S\in\Pi_{\varphi_{i+1}},\\
      1+\ell(H^{-1}(S)),&S\notin\Pi_{\varphi_i}[1]\cup\Pi_{\varphi_{i+1}}.
    \end{cases}
  \end{equation}
  To verify \eqref{eq:adjacent-form-signs}, we first determine the signs of $f$ and $g$ on objects of $\mathcal H_{i+1}$. If $S_1,\ldots,S_r$ are the composition factors of $E\in\mathcal H_{i+1}$, counted with multiplicity, then
  \begin{equation}
    h\bigl(\iota_{i+1}([E]_{\mathcal H_{i+1}})\bigr) =\sum_{a=1}^rh\bigl(\iota_{i+1}([S_a]_{\mathcal H_{i+1}})\bigr) \qquad(h\in\{f,g\}).
  \end{equation}
  Each summand is nonnegative. Since $\Pi_{\varphi_i}[1]$ and $\Pi_{\varphi_{i+1}}$ are Serre, \eqref{eq:adjacent-f-simple-values} and \eqref{eq:adjacent-g-simple-values} imply
  \begin{equation}\label{eq:adjacent-heart-signs}
    \begin{aligned}
      f\bigl(\iota_{i+1}([E]_{\mathcal H_{i+1}})\bigr)&\geq0, &\text{with equality iff }E\in\Pi_{\varphi_i}[1],\\
      g\bigl(\iota_{i+1}([E]_{\mathcal H_{i+1}})\bigr)&\geq0, &\text{with equality iff }E\in\Pi_{\varphi_{i+1}}.
    \end{aligned}
  \end{equation}

  Let $1\leq j\leq n$ and $0\ne X\in\Pi_{\varphi_j}$. If $j\geq i+1$, then $X\in\mathcal H_{i+1}$ and
  \begin{equation}
    \iota_{i+1}([X]_{\mathcal H_{i+1}})=[X]_\cA.
  \end{equation}
  Applying \eqref{eq:adjacent-heart-signs} to $X$ gives $f(X)>0$, while $g(X)=0$ for $j=i+1$ and $g(X)>0$ for $j>i+1$.

  If $j\leq i$, then $X[1]\in\mathcal H_{i+1}$ and
  \begin{equation}
    \iota_{i+1}([X[1]]_{\mathcal H_{i+1}})=-[X]_\cA.
  \end{equation}
  Applying \eqref{eq:adjacent-heart-signs} to $X[1]$ gives $g(X)<0$, while $f(X)=0$ for $j=i$ and $f(X)<0$ for $j<i$. This proves \eqref{eq:adjacent-form-signs}; \eqref{eq:adjacent-numerical-cuts} holds by the definition of $g$.
\end{proof}

\subsection{\texorpdfstring{A criterion for stability data when $\lvert\Phi\rvert\leq6$}{A criterion for stability data with at most six phases}}\label{subsec:small-phase-induction}

For $1\leq n\leq6$ and all $\Pi_\varphi$ wide, we verify condition (3) of Theorem~\ref{thm:positive-class-criterion} with $\bF=\bQ$. The next lemma extends a homomorphism across at most two boundary semistable subcategories.

\begin{lemma}\label{lem:two-fringe-extension}
  Let $(\Phi,\{\Pi_\varphi\}_{\varphi\in\Phi})$ be a stability data on $\cA$, with $\Phi=\{\varphi_1<\cdots<\varphi_n\}$ and every $\Pi_\varphi$ wide. Let $r\in\{0,1,2\}$ with $r<n$. If $f:K_0(\cA)\to\bZ$ satisfies $f(X)<0$ for $1\leq j\leq r$ and $0\ne X\in\Pi_{\varphi_j}$, there are homomorphisms $\theta_1,\ldots,\theta_r:K_0(\cA)\to\bQ$ such that, with $\theta_{r+1}=f$,
  \begin{equation}
    \theta_j(H_j)=0,\qquad \theta_j>\theta_{j+1} \quad(1\leq j\leq r).
  \end{equation}
  Dually, if $g:K_0(\cA)\to\bZ$ satisfies $g(X)>0$ for $n-r<j\leq n$ and $0\ne X\in\Pi_{\varphi_j}$, there are homomorphisms $\theta_{n-r+1},\ldots,\theta_n:K_0(\cA)\to\bQ$ such that, with $\theta_{n-r}=g$,
  \begin{equation}
    \theta_j(H_j)=0,\qquad \theta_{j-1}>\theta_j \quad(n-r<j\leq n).
  \end{equation}
  For $r=0$, both conclusions are vacuous.
\end{lemma}

\begin{proof}
  It suffices to prove the left-hand assertion; the right-hand one follows in $\cA^{\mathrm{op}}$ by reversing the phase order and negating the homomorphisms. The case $r=0$ is immediate.

  We first establish a rational-valued version of the case $r=1$, which will also be used when $r=2$.
  \begin{claim}
    For every homomorphism $h:K_0(\cA)\to\bQ$ satisfying $h(X)<0$ for $0\ne X\in\Pi_{\varphi_1}$, there is a homomorphism $\eta:K_0(\cA)\to\bQ$ with $\eta(H_1)=0$ and $\eta>h$.
  \end{claim}

  \begin{proof}[Proof of the claim]
    Define $\eta:K_0(\cA)\to\bQ$ on the basis $\{[S]:S\in\operatorname{Sim}\cA\}$ by
    \begin{equation}
      \eta(S)=
      \begin{cases}
        0,&S\in\Pi_{\varphi_1},\\
        h(S)+1,&S\notin\Pi_{\varphi_1}.
      \end{cases}
    \end{equation}
    Then $\eta>h$. By Lemma~\ref{lem:extreme slices are Serre}, $\Pi_{\varphi_1}$ is Serre, so $\eta(H_1)=0$.
  \end{proof}

  For $r=1$, apply the claim with $h=f$ and set $\theta_1=\eta$.

  For $r=2$, we have $n\geq3$. By Lemma~\ref{lem:fringe-simple-quotients}(\ref{item:fringe-simple-quotients}), each simple $T$ of $\Pi_{\varphi_2}$ which is not ambient simple fits into $0\to K_T\to T\to S_T\to0$, where $K_T\in\Pi_{\varphi_1}$, $\ell(K_T)\leq n-2$, and $S_T$ is an ambient simple of phase greater than $\varphi_2$. Moreover, by Lemma~\ref{lem:fringe-head-uniqueness}, $S_T\simeq S_{T'}$ implies $T\simeq T'$. Define on ambient simples
  \begin{equation}\label{eq:left-fringe-forms}
    \theta_2(S)=
    \begin{cases}
      f(S)+\dfrac{1}{n-1},&S\in\Pi_{\varphi_1},\\
      0,&S\in\Pi_{\varphi_2},\\
      -f(K_T)-\dfrac{\ell(K_T)}{n-1},&S\simeq S_T,\\
      f(S)+1,&\text{otherwise}.
    \end{cases}
  \end{equation}
  Since $S_T\simeq S_{T'}$ implies $T\simeq T'$, $\theta_2$ is well-defined. \revision{Since $f(T)$ is a negative integer, $-f(T)\geq1$. The bound $\ell(K_T)\leq n-2$ gives the following estimate for every $T$, without requiring finitely many ambient simples:}
  \begin{equation}
    \theta_2(S_T)-f(S_T) =-f(T)-\frac{\ell(K_T)}{n-1} \geq1-\frac{n-2}{n-1}>0.
  \end{equation}
  Together with $f(S)\leq-1$ for $S\in\operatorname{Sim}\cA\cap\Pi_{\varphi_2}$, \eqref{eq:left-fringe-forms} gives $\theta_2>f$. Since $\Pi_{\varphi_1}$ is Serre, $\theta_2(K_T)=f(K_T)+\ell(K_T)/(n-1)=-\theta_2(S_T)$. Thus $\theta_2(T)=0$ for every simple $T$ of $\Pi_{\varphi_2}$, and $\theta_2(H_2)=0$. Since $\Pi_{\varphi_1}$ is Serre and $f(S)\leq-1$ for $S\in\operatorname{Sim}\cA\cap\Pi_{\varphi_1}$, every $0\ne X\in\Pi_{\varphi_1}$ satisfies
  \begin{equation}
    \theta_2(X) =\sum_{S\in\operatorname{Sim}\cA}m_S(X)\theta_2(S) \leq-\frac{n-2}{n-1}\ell(X)<0.
  \end{equation}
  Applying the claim with $h=\theta_2$ and setting $\theta_1=\eta$ gives $\theta_1>\theta_2>f$ and $\theta_1(H_1)=0$.
\end{proof}

\revision{Combining Lemmas~\ref{lem:adjacent-numerical-cuts} and~\ref{lem:two-fringe-extension} covers every phase set of size at most six.}

\begin{samepage}
\begin{theorem}\label{thm:small-phase-induction}\label{thm:three-phase-wide-criterion}\label{thm:four-phase-wide-criterion}\label{thm:five-phase-module-criterion}\label{thm:six-phase-module-criterion}
  Let $\cA$ be an essentially small abelian length category, and let $(\Phi,\{\Pi_\varphi\}_{\varphi\in\Phi})$ be a stability data on $\cA$, where $\Phi=\{\varphi_1<\cdots<\varphi_n\}$ and $1\leq n\leq6$. There exists a stability function $\mu:\operatorname{Obj}^*(\cA)\to\review{\bQ}$ inducing $(\Phi,\{\Pi_\varphi\}_{\varphi\in\Phi})$ if and only if every $\Pi_\varphi$ is wide. More precisely, when every $\Pi_\varphi$ is wide, there are group homomorphisms $\theta_1,\ldots,\theta_n:K_0(\cA)\to\bQ$ satisfying
  \begin{equation}\label{eq:five-module-decreasing-forms}
    \begin{aligned}
      \theta_i(H_i)&=0&&\quad(1\leq i\leq n),\\
      \theta_i([S])&>\theta_{i+1}([S]) &&\quad(1\leq i<n,\ S\in\operatorname{Sim}\cA).
    \end{aligned}
  \end{equation}
\end{theorem}
\end{samepage}

\begin{proof}
  Necessity is Proposition~\ref{prop:wideness of stability data}. Assume every semistable subcategory is wide. For $n=1$, take $\theta_1=0$. For $2\leq n\leq6$, put
  \begin{equation}
    a=\left\lfloor\frac{n+1}{2}\right\rfloor,\qquad b=\left\lceil\frac{n+1}{2}\right\rceil,\qquad r=a-1=n-b\leq2.
  \end{equation}
  Since $a<n$ and $b-a\in\{0,1\}$, Lemma~\ref{lem:adjacent-numerical-cuts} with $i=a$ gives homomorphisms $f,g:K_0(\cA)\to\bZ$ such that $f-g=\ell$ and, for every $1\leq j\leq n$ and $0\ne X\in\Pi_{\varphi_j}$,
  \begin{equation}
    \renewcommand{\arraystretch}{1.2}
    \begin{array}{c|cccc}
      &j<a&j=a&j=a+1&j>a+1\\\hline
      f(X)&<0&=0&>0&>0\\
      g(X)&<0&<0&=0&>0
    \end{array}.
  \end{equation}
  In particular, $f(H_a)=g(H_{a+1})=0$ and $f(M)-g(M)=\ell(M)>0$ for $0\ne M\in\cA$. Put
  \begin{equation}
    \theta_j=f-(j-a)\ell\qquad(a\leq j\leq b).
  \end{equation}
  Then $\theta_j(H_j)=0$ for $a\leq j\leq b$, $\theta_j-\theta_{j+1}=\ell>0$ for $a\leq j<b$, and
  \begin{equation}
    \begin{aligned}
      \theta_a(X)&<0&&\quad(1\leq j<a,\ 0\ne X\in\Pi_{\varphi_j}),\\
      \theta_b(X)&>0&&\quad(b<j\leq n,\ 0\ne X\in\Pi_{\varphi_j}).
    \end{aligned}
  \end{equation}
  Since $a-1=n-b=r\leq2$, Lemma~\ref{lem:two-fringe-extension} extends $(\theta_j)_{a\leq j\leq b}$ to homomorphisms $\theta_1,\ldots,\theta_n:K_0(\cA)\to\bQ$ satisfying \eqref{eq:five-module-decreasing-forms}.

  This verifies Theorem~\ref{thm:positive-class-criterion}(3) with $\bF=\bQ$. Lemma~\ref{lem:decreasing-cuts-induction} gives an inducing stability function with values in $\bQ\subseteq\bR$.
\end{proof}

\subsection{Sharpness of the six-phase bound}\label{subsec:seven-phase-central}

The bound in Theorem~\ref{thm:small-phase-induction} is sharp: the following $D_4$ example has seven semisimple wide semistable subcategories but is not induced by any stability function. Any inducing stability function would assign two distinct phase values to the same object.

\begin{example}\label{ex:d4-seven-phase-counterexample}
  Let $K=\bC$, and take $\cA=\mod KQ$, where $Q$ is the quiver
  \begin{equation}
    \begin{tikzcd}[column sep=small,row sep=small]
      3\arrow[r]&2\arrow[r]&1\\
      &4\arrow[u]&
    \end{tikzcd}
  \end{equation}
  Write $i$ for the simple $KQ$-module at vertex $i$. We display modules by their radical layers, with their tops in the bottom row. Except for $X=\radlayers{1\\2\;2\\3\;4}$, each displayed module has vertex spaces of dimension at most one: the nonzero spaces are $K$, and the arrows between them are identities. For $X$, take
  \begin{equation}\label{eq:d4-central-representation}
    \begin{gathered}
      (X_1,X_2,X_3,X_4)=(K,K^2,K,K),\\[4pt]
      X_{2\to1}=(1\ \ 1),\qquad X_{3\to2}=\binom10,\qquad X_{4\to2}=\binom01.
    \end{gathered}
  \end{equation}
  Take $\Phi=\{\varphi_1<\cdots<\varphi_7\}$. The twelve indecomposables given by Gabriel's theorem \cite{Gabriel1972} have the following Auslander--Reiten quiver. Define $\Pi_{\varphi_i}$ as the additive subcategory generated by the vertices in the column labeled $\varphi_i$.
  \begin{equation}\label{eq:d4-seven-slices}
    \begin{tikzpicture}[x=1.65cm,y=1.05cm,baseline=(current bounding box.center),
      every node/.style={inner sep=3pt},>=stealth]
      \foreach \i in {1,...,7}
        \node at ({\i-1},1.65) {$\varphi_{\i}$};
      \node (s1) at (0,0) {$1$};
      \node (m12) at (1,0) {$\radlayers{1\\2}$};
      \node (m123) at (2,1) {$\radlayers{1\\2\\3}$};
      \node (s2) at (2,0) {$2$};
      \node (m124) at (2,-1) {$\radlayers{1\\2\\4}$};
      \node (x) at (3,0) {$\radlayers{1\\2\;2\\3\;4}$};
      \node (m24) at (4,1) {$\radlayers{2\\4}$};
      \node (m1234) at (4,0) {$\radlayers{1\\2\\3\;4}$};
      \node (m23) at (4,-1) {$\radlayers{2\\3}$};
      \node (m234) at (5,0) {$\radlayers{2\\3\;4}$};
      \node (s3) at (6,1) {$3$};
      \node (s4) at (6,-1) {$4$};
      \draw[->] (s1)--(m12);
      \foreach \v in {m123,s2,m124} {
        \draw[->] (m12)--(\v);
        \draw[->] (\v)--(x);
      }
      \foreach \v in {m24,m1234,m23} {
        \draw[->] (x)--(\v);
        \draw[->] (\v)--(m234);
      }
      \draw[->] (m234)--(s3);
      \draw[->] (m234)--(s4);
    \end{tikzpicture}
  \end{equation}

  \par\normalfont
  \revision{There is no path in the Auslander--Reiten quiver from a later column to an earlier one. Hence $\hom(\Pi_{\varphi_j},\Pi_{\varphi_i})=0$ for $j>i$. The vertices in each column are Hom-orthogonal bricks. AR duality gives $\ext^1(E,F)=0$ whenever $E,F$ lie in the same column.} Consequently each $\Pi_{\varphi_i}$ is semisimple and wide.

  Every object has a decomposition $M\simeq\bigoplus_{i=1}^7M_i$ with $M_i\in\Pi_{\varphi_i}$. The split filtration
  \begin{equation}\label{eq:d4-split-hn}
    0\subseteq M_7\subseteq M_7\oplus M_6\subseteq\cdots
    \subseteq M_7\oplus\cdots\oplus M_1\simeq M
  \end{equation}
  is an HN filtration after repeated terms are removed. Hence $(\Phi,\{\Pi_{\varphi_i}\}_{i=1}^7)$ is a stability data.

  \revision{The system $(\Phi,\{\Pi_{\varphi_i}\}_{i=1}^7)$ is not induced by any stability function.} Set $Z=1\oplus4$. In $K_0(\cA)$,
  \begin{equation}\label{eq:d4-coordinate-exact-rows}
    [Z]=\left[\radlayers{1\\2\\4}\right]-[2]
    =\left[\radlayers{1\\2\\3\;4}\right]-\left[\radlayers{2\\3}\right].
  \end{equation}
  \review{Suppose that $\mu$ induces this stability data, with value $t_i$ on $\Pi_{\varphi_i}\setminus\{0\}$, where $t_1<\cdots<t_7$. By Lemma~\ref{lem:stability-splitting-neutrality}, objects with the same class in $K_0(\cA)$ have the same $\mu$-value. Therefore
  \begin{equation}
    \mu(2\oplus Z)=\mu\left(\radlayers{1\\2\\4}\right)=\mu(2)=t_3.
  \end{equation}
  \begin{equation}
    \mu\left(\radlayers{2\\3}\oplus Z\right)
    =\mu\left(\radlayers{1\\2\\3\;4}\right)
    =\mu\left(\radlayers{2\\3}\right)=t_5.
  \end{equation}
  For each $A\in\{2,\radlayers{2\\3}\}$, the split exact sequence
  \begin{equation}
    0\longrightarrow A\longrightarrow A\oplus Z\longrightarrow Z\longrightarrow0
  \end{equation}
  has $\mu(A)=\mu(A\oplus Z)$ by the equalities above. Strict see-saw therefore forces $\mu(Z)=\mu(A)$. The two choices of $A$ give $\mu(Z)=t_3$ and $\mu(Z)=t_5$, contradicting $t_3<t_5$.}
  For the full phase chain $\varphi_1<\cdots<\varphi_7$, the tuple $(\mathbf x_0,\ldots,\mathbf x_7)=(\mathbf0,\mathbf0,\mathbf0,[Z],[Z],\mathbf0,\mathbf0,\mathbf0)$ is a nonzero solution of \eqref{eq:positive-class-flow}, since $0\ne[Z]\in H_3\cap H_5\cap K_0(\cA)_{\bQ_{\geq0}}$. Thus condition~(2) of Theorem~\ref{thm:positive-class-criterion} fails.
\end{example}

Thus seven is the least number of wide semistable subcategories for which induction by a stability function can fail.

\begin{remark}\label{rem:seven-phase-extension}
  For every finite $n>7$, work in $\mod(KQ\times K^{n-7})$ and append to the seven subcategories of Example~\ref{ex:d4-seven-phase-counterexample} the subcategories $\add T_1,\ldots,\add T_{n-7}$, where $T_j$ is supported on the $j$th $K$-factor. The resulting ordered subcategories are wide and give a stability data with HN filtrations obtained by direct sums. Any inducing stability function would restrict to one for the seven-phase example, so no such function exists.
\end{remark}

\section{\texorpdfstring{\revision{A counterexample to the red-path conjecture}}{A counterexample to the red-path conjecture}}\label{sec:mgs-red-paths}

\revision{We first recall some notations and results from \cite{BST}. We then prove a converse to their red-path theorem at the level of stability data and give a counterexample to their red-path conjecture.}

\begingroup
Assume that $K_0(\cA)$ has finite rank $d$. Choose simple representatives $S_1,\ldots,S_d$ and identify $K_0(\cA)$ with $\bZ^d$ using the basis $[S_1],\ldots,[S_d]$. The $j$th coordinate of $[M]$ is the composition multiplicity $m_{S_j}(M)$. A \emph{red path} is a continuous map $\boldsymbol\gamma:[0,1]\to\bR^d$ with
\begin{equation}\label{eq:red-path-endpoints}
  \boldsymbol\gamma(0)=\mathbf1,\qquad
  \boldsymbol\gamma(1)=-\mathbf1,\qquad \mathbf1=(1,\ldots,1),
\end{equation}
such that, for every nonzero object $M$, the real-valued function
\begin{equation}
  \rho_M(t)=\langle\boldsymbol\gamma(t),[M]\rangle,\qquad t\in[0,1],
\end{equation}
has exactly one zero, denoted by $t_M$. Here $\langle\cdot,\cdot\rangle$ denotes the standard inner product on $\bR^d$.  

The following theorem associates a stability function to every red path.

\begin{theorem}[Br\"ustle--Smith--Treffinger, Theorem~4.8]\label{thm:bst-red-path-stability}
  Let $\cA$ be an abelian length category with $\operatorname{rank}K_0(\cA)=d<\infty$. Every red path $\boldsymbol\gamma:[0,1]\to\bR^d$ induces a stability function
  \begin{equation}
    \mu_{\boldsymbol\gamma}:\operatorname{Obj}^*(\cA)\longrightarrow[0,1],\qquad
    \mu_{\boldsymbol\gamma}(M)=t_M,
  \end{equation}
  where $t_M$ is the unique zero of $t\mapsto\langle\boldsymbol\gamma(t),[M]\rangle$.
\end{theorem}

Saying that $\boldsymbol\gamma$ \emph{induces} \eqref{eq:mgs-definition} means that $\mu_{\boldsymbol\gamma}$ induces it in the sense of Section~2, with $p\in[0,1]$ in \eqref{eq:mgs-function-torsion}. Since no nonzero object has zero time $0$ or $1$, \revision{the target $[0,1]$ already provides both end torsion classes.}
\par\endgroup

\revision{We now prove the converse at the level of stability data.}

\begingroup
\begin{corollary}\label{cor:finite-induction-red-path}
  \revision{Let $(\Phi,\{\Pi_\varphi\}_{\varphi\in\Phi})$ be a stability data on $\cA$ with finitely many phases, and assume $K_0(\cA)$ has finite rank $d$. Then $(\Phi,\{\Pi_\varphi\}_{\varphi\in\Phi})$ is induced by a stability function into a totally ordered set if and only if it is induced by a red path. When an inducing stability function exists, an inducing red path may be chosen piecewise linear with rational vertices and strictly decreasing coordinates.}
\end{corollary}

\begin{proof}
  Write $\Phi=\{\varphi_1<\cdots<\varphi_n\}$ and choose ambient simple representatives $S_1,\ldots,S_d$. \revision{If $(\Phi,\{\Pi_\varphi\}_{\varphi\in\Phi})$ is induced by a stability function,} Corollary~\ref{cor:finite-phase-integral-separators} gives integer-valued homomorphisms $\theta_1,\ldots,\theta_n$ with $\theta_i(H_i)=0$ and $\theta_i(S_j)>\theta_{i+1}(S_j)$. Choose a positive integer $c>\max_{i,j}|\theta_i(S_j)|$ and put $\theta_0=c\ell$, $\theta_{n+1}=-c\ell$. Define
  \begin{equation}\label{eq:finite-induction-red-path}
    \boldsymbol\gamma\!\left(\frac{i}{n+1}\right)
    =\frac1c\bigl(\theta_i(S_1),\ldots,\theta_i(S_d)\bigr)
    \qquad(0\leq i\leq n+1),
  \end{equation}
  and interpolate linearly between consecutive vertices. \revision{The vertices in \eqref{eq:finite-induction-red-path} have rational coordinates,} and every coordinate strictly decreases from $1$ to $-1$. For every nonzero $M$, $\rho_M(t)=\langle\boldsymbol\gamma(t),[M]\rangle$ therefore strictly decreases from $\ell(M)$ to $-\ell(M)$ and has exactly one zero. Thus $\boldsymbol\gamma$ is a red path. If $0\ne E\in\Pi_{\varphi_i}$, its zero is $i/(n+1)$, so Lemma~\ref{lem:HN-recognition} shows that $\mu_{\boldsymbol\gamma}$ \revision{induces $(\Phi,\{\Pi_\varphi\}_{\varphi\in\Phi})$.} Conversely, every red path yields a stability function by Theorem~\ref{thm:bst-red-path-stability}.
\end{proof}

\begin{corollary}\label{cor:short-mgs-red-path}
  Every maximal green sequence of length at most six in an essentially small abelian length category of finite rank is induced by a red path. The path may be chosen piecewise linear with rational vertices and strictly decreasing coordinates.
\end{corollary}

\begin{proof}
  \revision{Let $0=\cT_0\subsetneq\cdots\subsetneq\cT_m=\cA$ be a maximal green sequence. The corresponding complete Hom-orthogonal brick sequence determines a stability data $(\Phi,\{\Pi_\varphi\}_{\varphi\in\Phi})$ with $m$ wide semistable subcategories and finite HN filtrations, by \cite[Theorem~2.13]{LiLiu2020} and \cite[Theorem~3.17]{LQInfiniteMGS}. If $m\leq6$, Theorem~\ref{thm:small-phase-induction} gives an inducing stability function for $(\Phi,\{\Pi_\varphi\}_{\varphi\in\Phi})$. Corollary~\ref{cor:finite-induction-red-path} gives a red path $\boldsymbol\gamma$ inducing $(\Phi,\{\Pi_\varphi\}_{\varphi\in\Phi})$. The distinct torsion classes induced by $\mu_{\boldsymbol\gamma}$ are $\cT_0,\ldots,\cT_m$, by the finite HN filtrations and the endpoint convention of Section~2.}
\end{proof}
\par\endgroup

\begingroup
Their red-path conjecture is the following.

\begin{conjecture}[Br\"ustle--Smith--Treffinger, Conjecture~4.14]\label{conj:bst-red-path}
  Let $\cA$ be an abelian length category of finite rank. Every maximal green sequence in $\cA$ is induced by a red path.
\end{conjecture}

We disprove Conjecture~\ref{conj:bst-red-path} in Example~\ref{thm:d4-mgs-not-induced}, using the $D_4$ algebra of Example~\ref{ex:d4-seven-phase-counterexample}. In this example, the coordinate order is the vertex order $(1,2,3,4)$.
\par\endgroup

\begingroup
\begin{example}\label{thm:d4-mgs-not-induced}\label{cor:bst-red-path-counterexample}
  Let $\cA=\mod KQ$ be as in Example~\ref{ex:d4-seven-phase-counterexample}. With the same notation, set
  \begin{equation}\label{eq:d4-mgs-bricks}
    (B_1,\ldots,B_{11})
    =\left(1,\radlayers{1\\2},\radlayers{1\\2\\4},
      \radlayers{1\\2\\3},2,\radlayers{1\\2\;2\\3\;4},
      \radlayers{2\\4},\radlayers{1\\2\\3\;4},
      4,\radlayers{2\\3},3\right).
  \end{equation}
  For $\Phi=\{\varphi_1<\cdots<\varphi_{11}\}$, put $\Pi_{\varphi_j}=\add B_j$ and
  \begin{equation}\label{eq:d4-mgs-torsion-chain}
    \cT_0=0,\qquad
    \cT_i=\filt(\Pi_{\varphi_{12-i}},\ldots,\Pi_{\varphi_{11}})
    \quad(1\leq i\leq11).
  \end{equation}
  Then $(\Phi,\{\Pi_{\varphi_j}\}_{j=1}^{11})$ is a stability data with semisimple wide semistable subcategories. The chain \eqref{eq:d4-mgs-torsion-chain} is a maximal green sequence of length $11$ that is not induced by any stability function into a totally ordered set. In particular, it is not induced by a red path, so Conjecture~\ref{conj:bst-red-path} is false.
\end{example}

\begin{proof}
  The sequence $(B_1,\ldots,B_{11})$ is a complete backward Hom-orthogonal (CBHO) sequence of bricks in the sense of \cite[Definition~3.15]{LQInfiniteMGS}. Indeed, the Hom vanishings follow from Example~\ref{ex:d4-seven-phase-counterexample}, together with $\hom(B_{10},B_9)=\hom(\radlayers{2\\3},4)=0$. The only omitted indecomposable is $\radlayers{2\\3\;4}$. The nonzero maps $B_{10}\to\radlayers{2\\3\;4}\to B_9$ force any insertion of this module to lie after $B_{10}$ and before $B_9$, which is impossible.

  By \cite[Theorem~3.17]{LQInfiniteMGS}, $(\Phi,\{\Pi_{\varphi_j}\}_{j=1}^{11})$ is a finest system of stability data, and \eqref{eq:d4-mgs-torsion-chain} is the corresponding maximal green sequence, of length $11$.

  Suppose that a stability function $\mu$ induces this chain.
  By \cite[Theorems~3.5 and~3.17]{LQInfiniteMGS}, the stability
  data induced by $\mu$ is finest and agrees, up to an
  order-preserving relabelling of phases, with
  $(\Phi,\{\Pi_{\varphi_j}\}_{j=1}^{11})$.
  Since $B_j\in\Pi_{\varphi_j}$, we obtain
  \begin{equation}\label{eq:d4-mgs-forced-order}
  	\mu(B_1)<\mu(B_2)<\cdots<\mu(B_{11}).
  \end{equation}

  Put $Z=1\oplus3$. The dimension vectors give
  \begin{equation}\label{eq:d4-mgs-class-obstruction}
    [B_4]=[B_5]+[Z],\qquad [B_8]=[B_7]+[Z].
  \end{equation}
  By Lemma~\ref{lem:stability-splitting-neutrality},
  \begin{equation}
    \mu(B_4)=\mu(B_5\oplus Z),\qquad
    \mu(B_8)=\mu(B_7\oplus Z).
  \end{equation}
  Applying strict see-saw to the corresponding split exact sequences and using \eqref{eq:d4-mgs-forced-order}, we obtain $\mu(Z)<\mu(B_4)$ and $\mu(B_8)<\mu(Z)$. Hence
  \begin{equation}
    \mu(Z)<\mu(B_4)<\mu(B_8)<\mu(Z),
  \end{equation}
  a contradiction.

  Thus no stability function into any totally ordered set induces this chain. By Theorem~\ref{thm:bst-red-path-stability}, no red path induces it either.
\end{proof}

\begin{remark}
  The contradiction uses the decomposable object $Z=1\oplus3$. The unique-zero condition for a red path applies to every nonzero object, including such direct sums.
\end{remark}
\par\endgroup

\bibliographystyle{plain}
\bibliography{references}

\end{document}